\documentclass[a4paper,12pt]{amsart}

\usepackage{amssymb}
\usepackage{latexsym}
\usepackage{amsmath}
\usepackage{amscd}
\usepackage{graphicx}
\usepackage{palatino}

\title[Formal symplectic geometry and moduli spaces]{Computations in formal symplectic geometry and
characteristic classes of moduli spaces}

\author{Shigeyuki Morita}
\address{Graduate School of Mathematical Sciences, 
University of Tokyo, 
3-8-1 Komaba, 
Meguro-ku, Tokyo, 153-8914, Japan}
\email{morita@ms.u-tokyo.ac.jp}
\author{Takuya Sakasai}
\address{Graduate School of Mathematical Sciences, 
University of Tokyo, 
3-8-1 Komaba, 
Meguro-ku, Tokyo, 153-8914, Japan}
\email{sakasai@ms.u-tokyo.ac.jp}
\author{Masaaki Suzuki}
\address{Department of Frontier Media Science, Meiji University, 
4-21-1 Nakano, 
Nakano-ku, Tokyo, 164-8525, Japan}
\email{macky@fms.meiji.ac.jp}


\newtheorem{thm}{Theorem}[section]
\newtheorem{prop}[thm]{Proposition}

\newtheorem{cor}[thm]{Corollary}
\theoremstyle{definition}
\newtheorem{definition}[thm]{Definition}
\newtheorem{example}[thm]{Example}
\newtheorem{remark}[thm]{Remark}
\newtheorem{problem}[thm]{Problem}
\newtheorem{conj}[thm]{Conjecture}

\begin{document}

\newcommand{\Mg}{\mathcal{M}_g}
\newcommand{\Mgp}{\mathcal{M}_{g,\ast}}
\newcommand{\Mgb}{\mathcal{M}_{g,1}}

\newcommand{\hg}{\mathfrak{h}_{g,1}}
\newcommand{\ag}{\mathfrak{a}_g}
\newcommand{\Ln}{\mathcal{L}_n}

\newcommand{\Sg}{\Sigma_g}
\newcommand{\Sgb}{\Sigma_{g,1}}
\newcommand{\la}{\lambda}

\newcommand{\Symp}[1]{Sp(2g,\mathbb{#1})}
\newcommand{\symp}[1]{\mathfrak{sp}(2g,\mathbb{#1})}
\newcommand{\gl}[1]{\mathfrak{gl}(n,\mathbb{#1})}

\newcommand{\At}[1]{\mathcal{A}_{#1}^t (H)}
\newcommand{\Hq}{H_{\mathbb{Q}}}

\newcommand{\Ker}{\mathop{\mathrm{Ker}}\nolimits}
\newcommand{\Hom}{\mathop{\mathrm{Hom}}\nolimits}
\renewcommand{\Im}{\mathop{\mathrm{Im}}\nolimits}

\newcommand{\Der}{\mathop{\mathrm{Der}}\nolimits}
\newcommand{\Out}{\mathop{\mathrm{Out}}\nolimits}
\newcommand{\Aut}{\mathop{\mathrm{Aut}}\nolimits}
\newcommand{\Q}{\mathbb{Q}}
\newcommand{\Z}{\mathbb{Z}}
\newcommand{\R}{\mathbb{R}}

\begin{abstract}
We make explicit computations in the 
formal symplectic geometry of Kontsevich and determine the
Euler characteristics of the three cases, 
namely commutative, Lie and associative ones, 
up to certain weights.
From these, we obtain some non-triviality results in each case.
In particular, we determine the
{\it integral} Euler characteristics of the outer automorphism groups $\mathrm{Out}\, F_n$ 
of free groups for all $n\leq 10$ and prove the existence of
plenty of rational cohomology classes of {\it odd} degrees.
We also clarify the relationship of the 
commutative graph homology with finite type invariants
of homology $3$-spheres as well as
the leaf cohomology classes
for transversely symplectic foliations.
Furthermore we prove the existence of several new
{\it non-trivalent} graph homology classes of {\it odd} degrees.
Based on these computations, 
we propose a few conjectures and problems on 
the graph homology and the
characteristic classes of the moduli spaces of graphs as well as curves.

\end{abstract}

\renewcommand\baselinestretch{1.1}
\setlength{\baselineskip}{16pt}

\newcounter{fig}
\setcounter{fig}{0}

\maketitle

\section{Introduction and statements of the main results}\label{sec:intro}

In celebrated papers \cite{kontsevich1}\cite{kontsevich2}, Kontsevich considered three infinite
dimensional Lie algebras, namely commutative, Lie and associative ones.
He proved that the stable homology group of each of
these Lie algebras is isomorphic to a free graded commutative algebra
generated by the stable homology group of $\mathfrak{sp}(2g,\Q)$
as $g$ tends to infinity together with certain set of generators which
he described explicitly. It is the totality of 
the {\it graph homologies} for the commutative case, the totality of 
the cohomology groups of the
outer automorphism groups of free groups,
denoted by $\mathrm{Out}\,F_n$,
for the Lie case, and 
the totality of the cohomology groups of the moduli spaces of curves
with {\it unlabeled} marked points, denoted by $\mathbf{M}_g^m/\mathfrak{S}_m$,
for the associative case.

As for the commutative (resp. associative) case,
Kontsevich described a general method of constructing
cycles of the corresponding graph complex 
by making use of finite dimensional Lie (resp. $A_\infty$) algebras 
with non-degenerate invariant scalar products. 
In the Lie case, however, he mentioned that no non-trivial class
was obtained by similar construction.
In the associative case, he also introduced a dual construction of
producing cocycles starting from a differential associative algebra 
with non-degenerate odd scalar product and trivial cohomology.
Certain detailed description and generalizations of these methods
have been given by several authors including
Hamilton, Lazarev (see e.g. \cite{hl}\cite{h}) and others. 

However there have been known only a few results which deduce 
new information about the graph homology and cohomology groups
of  $\mathrm{Out}\,F_n$ or $\mathbf{M}_g^m/\mathfrak{S}_m$
by making a direct use of the above theorem of Kontsevich.
First, as for the Lie case, in \cite{morita99} the first named author defined a series of
certain unstable homology classes of  $\mathrm{Out}\,F_n$
by using his trace maps introduced in \cite{morita93}.
Only the first three classes are known to be non-trivial (see Conant and Vogtmann \cite{cov} and Gray \cite{gr}).
Second, recently Conant, Kassabov and Vogtmann \cite{ckv} made a remarkable
new development in this direction and defined many more classes. 
Thirdly, in the commutative case,
the existence of two graph homology classes of {\it odd}
degrees was proved, one in
Gerlits \cite{ge} and the other in Conant, Gerlits and Vogtmann\cite{cgv}.
Fourthly, as for the associative case,
in \cite{morita08} a series of certain unstable homology classes for genus $1$ moduli
spaces was introduced, all of which have been proved to be non-trivial by Conant \cite{con}. 
Finally, in our recent paper \cite{mss1}
we determined the stable abelianization of the Lie algebra in the associative case.
As an application of this result, we obtained a new proof of an unpublished result of Harer.
Church, Farb and Putman \cite{cfp1} gave a different proof.

The purpose of this paper is to continue these lines of investigations.
We obtain new results in each of the three cases.

To be more precise, let $\Sigma_{g,1}$ be a compact oriented surface of genus $g\geq 1$ 
with one boundary component
and we denote its first rational homology group
$H_1(\Sigma_{g,1};\Q)$ simply by $H_\Q$. It can be regarded as the standard symplectic 
vector space of dimension $2g$ induced from the intersection 
pairing on it. Let $\mathfrak{c}_g$ denote the graded Lie algebra consisting of 
Hamiltonian polynomial vector fields, without constant terms, on $H_\Q\otimes\R\cong\R^{2g}$ with rational coefficients.  
The homogeneous degree $k$ part, denoted
by $\mathfrak{c}_g(k)$, can be naturally identified with $S^{k+2}H_\Q$
where $S^{k}H_\Q$ denotes the $k$-th symmetric power of $H_\Q$.
Let $\mathfrak{c}^+_g$ be the ideal of $\mathfrak{c}_g$ consisting of
Hamiltonian polynomial vector fields without linear terms.
Next, we denote by $\mathcal{L}_{H_\Q}$ the 
free Lie algebra generated by $H_\Q$. Let $\mathfrak{h}_{g,1}$
be the graded Lie algebra consisting of {\it symplectic} derivations of $\mathcal{L}_{H_\Q}$
and let $\mathfrak{h}^+_{g,1}$ be the ideal consisting of derivations 
with {\it positive} degrees.
This Lie algebra was introduced in the theory of Johnson homomorphisms 
before the work of Kontsevich (see \cite{morita89}) and has been investigated extensively. We use our notation
for this Lie algebra. The notation $\mathfrak{h}_g$ is reserved for the case of a closed surface
(see Remark \ref{rem:t})
while $\mathfrak{h}_{g,1}$ corresponds to genus $g$ compact surface with one boundary component.

Finally, let $T_0H_\Q$ denote the free associative algebra without unit generated by $H_\Q$.
Let $\mathfrak{a}_{g}$
be the graded Lie algebra consisting of {\it symplectic} derivations of $T_0H_\Q$
and let $\mathfrak{a}^+_{g}$ be the ideal consisting of derivations 
with {\it positive} degrees.

We denote by $\mathrm{Sp}(2g,\Q)$ the symplectic group which
we sometimes denote simply by $\mathrm{Sp}$. If we fix a symplectic 
basis of $H_\Q$, then the space  $H_\Q$ can be 
regarded as the standard representation of $\mathrm{Sp}(2g,\Q)$.
Each piece $\mathfrak{c}_g(k), \mathfrak{h}_{g,1}(k), \mathfrak{a}_{g}(k)$,
of the three graded Lie algebras,
is naturally an $\mathrm{Sp}$-module so that it has an irreducible decomposition. 
It is known that this decomposition stabilizes when $g$ is sufficiently large.

Now we describe our main results. We determine the dimensions of the chain complexes which compute the
$\mathrm{Sp}$-invariant stable homology of the three Lie algebras $\mathfrak{c}^+_g, \mathfrak{h}^+_{g,1}, \mathfrak{a}^+_g$
up to certain weights (see Tables \ref{tab:c}, \ref{tab:h} and \ref{tab:a}). From this, we determine the Euler characteristic of each 
weight summand and the result is given as follows. For the definition of {\it weight}, see Section $2$.

\begin{thm}
The Euler characteristics $\chi$
of the $\mathrm{Sp}$-invariant stable homologies of 
the three Lie algebras $\mathfrak{c}^+_g, \mathfrak{h}^+_{g,1}, \mathfrak{a}^+_g$
up to weight $20$, $18$ or $16$ are given as follows.
\begin{align*}
\mathrm{(i)}\ &  \chi(H_*(\mathfrak{c}^+_{\infty})^{\mathrm{Sp}}_{w})=1,2,3,6,8,14,20,32,44, 68\quad (w=2,4,\ldots,20)\\
\mathrm{(ii)}\ &   \chi(H_*(\mathfrak{h}^+_{\infty,1})^{\mathrm{Sp}}_{w})=1,2,4,6,10,16,23,13,-96\quad (w=2,4,\ldots,18)\\
\mathrm{(iii)}\ &   \chi(H_*(\mathfrak{a}^+_{\infty})^{\mathrm{Sp}}_{w})=2,5,12,24,50,100, 188, 347\quad (w=2,4,\ldots,16).
\end{align*}
\label{thm:chi}
\end{thm}

By combining Theorem \ref{thm:chi} above with the description of the
generators of the stable homologies due to Kontsevich, we obtain the following result.
Part $\mathrm{(ii)}$ proves, in particular, the existence of {\it odd} dimensional rational
homology classes of the outer automorphism groups of free groups for the {\it first} time.

\begin{thm}
The integral Euler characteristics $e$
of the primitive part of $\mathrm{Sp}$-invariant stable homologies of 
the three Lie algebras $\mathfrak{c}^+_g, \mathfrak{h}^+_{g,1}, \mathfrak{a}^+_g$,
up to weight $20$, $18$ or $16$, are given as follows.
\begin{align*}
\mathrm{(i)}\ &  e(G_*^{(n)})=1,1,1,2,1,2,2,2,1,3\quad (n=2,3,\ldots,11; w=2n-2)\\
\mathrm{(ii)}\ &  e(\mathrm{Out}\, F_{n})=1,1,2,1,2,1,1,-21,-124\quad  (n=2,3,\ldots,10; w=2n-2)\\
\mathrm{(iii)}\ &  \sum_{\begin{subarray}{c}
2g-2+m=n\\ m>0
\end{subarray}}
e(\mathbf{M}_g^m/\mathfrak{S}_m)=2,2,4,2,6,6, 6, 1\quad (n=1,2,\ldots,8; w=2n)
\end{align*}
\label{thm:chip}
\end{thm}
\noindent
Here $G_*^{(n)}$ denotes the graph complex due to Kontsevich
which is defined in terms of graphs with the Euler characteristic $(1-n)$, $\mathrm{Out}\,F_n$
denotes the outer automorphism group of the free group $F_n$ of rank $n$,
$\mathbf{M}_g^m$ denotes the moduli space of curves of genus $g$ with
labeled $m$ marked points and $\mathfrak{S}_m$ denotes the $m$-th symmetric group.

The commutative case of the graph homology has deep connections
with two important subjects in topology. One is the theory of {\it finite type}
invariants for homology $3$-spheres initiated by Ohtsuki \cite{oh}.
The other is the theory of characteristic classes of {\it transversely symplectic}
foliations going back to Gelfand, Kalinin and Fuks \cite{gkf} and more recently 
developed by Kontsevich \cite{kontsevich3} 
and further in \cite{metoki}\cite{km}.
On the other hand, a beautiful connection between these two theories
was found by Garoufalidis and Nakamura \cite{gn}.

Let $\mathcal{A}(\emptyset)$ denote the commutative algebra
generated by vertex oriented connected trivalent graphs 
modulo the two relations, one is the $\mathrm{(AS)}$ relation
and the other is the $\mathrm{(IHX)}$ relation. 
This algebra plays a fundamental role in the former theory above.
In fact, Le \cite{l} and Garoufalidis and Ohtsuki \cite{go} proved that 
the graded algebra associated with the Ohtsuki filtration on the space of all the 
homology $3$-spheres is isomorphic to $\mathcal{A}(\emptyset)$
which is a polynomial algebra generated by the
subspaces $\mathcal{A}^{(2n-2)}_{\mathrm{conn}} \ (n=2,3,\ldots)$ 
corresponding to connected graphs with $(2n-2)$ vertices.
Furthermore the completion $\widehat{\mathcal{A}}(\emptyset)$ of
$\mathcal{A}(\emptyset)$ with respect to its gradings serves as the target of the $\mathrm{LMO}$ invariant 
introduced in \cite{lmo}.

As is well-known,  the top homology group $H_{2n-2}(G_*^{(n)})$ of $G_*^{(n)}$
is canonically isomorphic to 
$\mathcal{A}^{(2n-2)}_{\mathrm{conn}}$ because it can be seen that the top coboundary operator
in the dual of the graph complex corresponds to the $\mathrm{(IHX)}$ relation
(see Proposition \ref{prop:ce}).
We can deduce from this fact  that $\mathcal{A}(\emptyset)$ can be
embedded into $H_*(\mathfrak{c}^+_\infty)^{\mathrm{Sp}}$ as a bigraded subalgebra.
We define $\mathcal{E}$ to be the {\it complementary} bigraded algebra (see Definition \ref{def:e}
for details) so that we have an isomorphism $H_*(\mathfrak{c}^+_\infty)^{\mathrm{Sp}}\cong\mathcal{A}(\emptyset)\otimes
\mathcal{E}$ of bigraded algebras. This bigraded algebra $\mathcal{E}$ can be interpreted as
the space of all the graph homology classes represented by {\it non-trivalent} graphs.
In the context of the theory of stable leaf cohomology classes for
transversely symplectic foliations, it can also be interpreted as the dual
space of all the {\it exotic} characteristic classes. 
Here by {\it exotic} we mean that the class depends on higher jets
than the connection as well as the curvature forms 
by which the usual secondary characteristic classes are defined.
See Section $5$ for details.

Now we can deduce the following result from Theorem \ref{thm:chip} (i).

\begin{thm}
There exists an isomorphism
$$
H_*(\mathfrak{c}^+_\infty)^{\mathrm{Sp}}\cong\mathcal{A}(\emptyset)\otimes
\mathcal{E}
$$
of bigraded algebras. If we denote by $P\mathcal{E}$
the primitive part of $\mathcal{E}$, then 
the Euler characteristic of its weight $w$-part
$P\mathcal{E}_w$ is given by
$$
e(P\mathcal{E}_w)=0,0,0,0,-1,-1,-2,-3,-5,-5\quad (w=2,4,\ldots,20).
$$
It follows that there exist several odd dimensional
non-trivalent  graph homology classes, as well as exotic stable leaf cohomology classes
for transversely symplectic foliations,
in each of the weights $w=10,12,\ldots,20$.
\label{thm:ce}
\end{thm}

The above theorem in the range $w\leq 10$ is essentially due to Gerlits \cite{ge}, 
Theorem 4.1, 
and the case $w=12$ is due to Conant, Gerlits and Vogtmann \cite{cgv}, Theorem 5.1.
In fact, in the former paper the author computed, among other things, $H_*(G_*^{(n)})$
for all $n\leq 6$ and the case $n=7$ was treated in the latter paper.
If we combine the former result with the above connection with
the theory of foliations, we can conclude the existence of a certain {\it exotic} 
stable leaf cohomology class of 
transversely symplectic foliations of degree $7$ 
and weight $10$.
This is the {\it first} appearance of such classes.

Next we consider Theorem \ref{thm:chip} $\mathrm{(ii)}$.
As is well-known, there is a beautiful formula for the {\it rational} Euler
characteristics of the mapping class groups due to Harer-Zagier \cite{hz}
and Penner \cite{p}. In the case of $\mathrm{Out}\,F_{n}$,
Smillie and Vogtmann \cite{sv} obtained a generating function for the
rational Euler characteristics $\chi(\mathrm{Out}\,F_{n})$ and computed them for all $n\leq 100$. 
They are all negative and they conjecture that they are always negative. 
However, compared to the case of the mapping class groups, there still
remain many open problems. For example, the relation between the rational
and the {\it integral} Euler characteristics seems to be not very well understood.
We will compare our computation above with the result of Smillie and Vogtmann
cited above in Table \ref{tab:chie} and we observe a very interesting
behavior of the two numbers for the first time.
See Section $6$ for details of this as well as 
other discussions of our results.

Finally we consider the case of $\mathfrak{a}_g$. In this case,
Kontsevich proved that the 
primitive part of $H_*(\mathfrak{a}^+_\infty)^{\mathrm{Sp}}$
corresponds to the totality of the 
$\mathfrak{S}_m$-invariant rational cohomology groups
$H^*(\mathbf{M}_{g}^m;\Q)^{\mathfrak{S}_m}$ of the moduli spaces
$\mathbf{M}_g^m$
of genus $g$ curves with $m$ marked points for all $g, m$
with $2g-2+m>0, m\geq 1$.
There have been known many results concerning the cohomology of these moduli
spaces for the cases of low genera $g=0,1,2,3,4$
due to Getzler-Kapranov \cite{gk}, Getzler \cite{getz98}\cite{getz99}, Looijenga \cite{loo}, 
Tommasi \cite{tom05}\cite{tom06}, Gorsky \cite{go1}\cite{go2}, Bergstr\"om \cite{b} and others.
In Section $7$, we will check that our computation of the Euler characteristics
is consistent with these known results
or deduced from them by explicit computations, in the range $2g-2+m\leq 8$.

In our forthcoming paper \cite{mss2a}, which is a sequel to this paper, we will
extend both of Theorem \ref{thm:chi} $\mathrm{(iii)}$ and Theorem \ref{thm:chip} $\mathrm{(iii)}$
from $w=16$ to $500$ by adopting a completely different method. More precisely,
we use a formula of  Gorsky \cite{go2} for the equivariant Euler characteristics of the
moduli space of curves to obtain certain closed formulas which enable us to determine
the above values. However, in this paper we only describe the values which we deduced from 
Table \ref{tab:a} in order to compare with the other two cases. See Section $7$ for
more details about this point.

To compute the graph homologies directly, we have to enumerate certain types of graphs.
However, according to the number of vertices increases, the difficulty of the problem of deciding 
whether two given graphs are isomorphic to each other grows very rapidly.
In view of this, we adopted a method in a pure framework of symplectic representation theory.
This has a disadvantage that the dimensions which we have to compute are considerably larger than the
graph theoretical method, because there is no effective way to distinguish between connected
and disconnected graphs in the framework of representation theory. 
In order to overcome this difficulty, we made various devices to lighten the 
burden imposed on computers. More precisely, our task is to determine the
dimensions of the $\mathrm{Sp}$-invariant subspaces of various $\mathrm{Sp}$-modules.
Theoretically there is no problem here because we know the characters of these modules
completely. However, the problem lies in the huge size of the data as well as the time
which computers need. We have developed several our own programs on the 
computer software Mathematica which realize theoretical considerations in the
representation theory. See Section $4$ for details.
We have also made an extensive use of the computer program LiE to obtain
irreducible decompositions of various $\mathrm{Sp}$-modules.

{\it Acknowledgement.} The authors would like 
to thank Christophe Soul\'e for helpful information
about the (non)vanishing of the Borel regulator classes. 
We also thank James Conant 
and Alastair Hamilton for informing the authors about their works \cite{cgv} and \cite{h}. 
Thanks are also due to a referee for pointing out that
Proposition \ref{prop:ce} is well-known.
The first named author would like to thank Richard Hain and Hiroaki Nakamura 
for enlightening discussions about the symplectic representation theory
related to the mapping class group
in the $1990$'s. 
The authors were partially supported by KAKENHI (No.~24740040 and 
No.~24740035), 
Japan Society for the Promotion of Science, 
Japan.

\section{Preliminaries}\label{sec:p}

In this section, we prepare a few facts about the (co)homology of graded Lie algebras
which will be needed in our later considerations.

Let $\mathfrak{g}=\oplus_{k=0}^{\infty}\mathfrak{g}(k)$ be a graded Lie algebra over $\Q$ and let
$\mathfrak{g}^+=\oplus_{k=1}^{\infty}\mathfrak{g}(k)$ be its ideal consisting of
all the elements of $\mathfrak{g}$ with {\it positive} gradings.
We assume that each piece $\mathfrak{g}(k)$ is finite dimensional for all $k$.
Then the chain complex $C_*(\mathfrak{g})$ of $\mathfrak{g}$
splits into the direct sum 
$$
C_*(\mathfrak{g})=\bigoplus_{w=0}^\infty C_*^{(w)}(\mathfrak{g})
$$
of {\it finite dimensional} subcomplexes 
$C_*^{(w)}(\mathfrak{g})=\oplus_{i=0}^w C_i^{(w)}(\mathfrak{g})$ where 
$$
C_i^{(w)}(\mathfrak{g})=
\bigoplus_{\begin{subarray}{c}
i_0+ i_1+\cdots + i_w=i\\ 
i_1+2 i_2+\cdots +w i_w=w
\end{subarray}}
\wedge^{i_0} (\mathfrak{g}(0))\otimes\wedge^{i_1} (\mathfrak{g}(1))\otimes\cdots\otimes
\wedge^{i_w} (\mathfrak{g}(w))
$$
so that $C_i^{(w)}(\mathfrak{g})=0$ for $i> w+\frac{1}{2} d(d-1)\ (d=\mathrm{dim}\, \mathfrak{g}(0))$.
This induces a bigraded structure on the homology group $H_*(\mathfrak{g})$ described as
$$
H_i(\mathfrak{g})=\bigoplus_{w=0}^\infty H_i(\mathfrak{g})_{w}
$$
where $H_i(\mathfrak{g})_{w}=H_i(C_*^{(w)}(\mathfrak{g}))$.
We call $H_i(\mathfrak{g})_{w}$ the {\it weight} $w$-part of $H_i(\mathfrak{g})$.
Let $\widehat{\mathfrak{g}}$ be the completion of $\mathfrak{g}$ with respect to the
grading and let $H^*_c(\widehat{\mathfrak{g}})$ be the {\it continuous} cohomology.
Then we have
$$
H^k_c(\widehat{\mathfrak{g}})\cong \bigoplus_{w=0}^\infty\ \left(H_k(\mathfrak{g})_w\right)^*.
$$

Now suppose that $\mathfrak{g}$ is an $\mathrm{Sp}$-graded Lie algebra by which
we mean that each piece $\mathfrak{g}(k)$ is a finite dimensional representation
of $\mathrm{Sp}(2g,\Q)$ for some fixed $g$ such that the bracket operation 
$\mathfrak{g}(i)\otimes\mathfrak{g}(j)\rightarrow\mathfrak{g}(i+j)$ is a morphism
of $\mathrm{Sp}$-modules for any $i,j$. We further assume that 
$\mathfrak{g}(0)=\mathfrak{sp}(2g,\Q)\cong S^2H_\Q$. Then we have a split extension
$$
0\rightarrow \mathfrak{g}^+\rightarrow \mathfrak{g}\rightarrow\mathfrak{sp}(2g,\Q)\rightarrow 0
$$
of Lie algebras. The $E^2$-term of the Hochschild-Serre spectral sequence for the homology 
of $\mathfrak{g}$ is given by
$$
E^2_{p,q}=H_p(\mathfrak{sp}(2g,\Q);H_q(\mathfrak{g}^+)).
$$
By the assumption, the chain complex $C_*(\mathfrak{g}^+)$ decomposes into
the direct sum of subcomplexes corresponding to $\mathrm{Sp}$-irreducible 
components. It follows that the homology group $H_q(\mathfrak{g}^+)$ also
decomposes into the $\mathrm{Sp}$-irreducible components.
In particular, we have the $\mathrm{Sp}$-invariant part which we denote by
$H_q(\mathfrak{g}^+)^{\mathrm{Sp}}$. 
Here we can apply the well-known
vanishing theorem (see Chevalley and Eilenberg \cite{ce}) to conclude that
$H_p(\mathfrak{sp}(2g,\Q);H_q(\mathfrak{g}^+)_{\lambda})=0$
for any $p \ge 0$ and for any $\mathrm{Sp}$-irreducible component
$\lambda$ different from the trivial representation.
It follows that the spectral sequence collapses
at the $E^2$-term and we have an isomorphism
$$
H_*(\mathfrak{g})\cong H_*(\mathfrak{sp}(2g;\Q))\otimes H_*(\mathfrak{g}^+)^{\mathrm{Sp}}.
$$
This argument can be applied in the case of three Lie algebras treated in this paper.
More precisely, we set $\mathfrak{g}_g$ to be one of $\mathfrak{c}_g, \mathfrak{h}_{g,1}$ or 
$\mathfrak{a}_g$. Then we have natural embeddings
$\mathfrak{g}_g\subset \mathfrak{g}_{g+1}$ so that we can consider the union
(or equivalently the direct limit)
$\mathfrak{g}_\infty=\lim_{g\to\infty}\mathfrak{g}_g$. Also we have its ideal
$\mathfrak{g}^+_\infty$. The homology groups of them are given by
$$
H_*(\mathfrak{g}_\infty)=\lim_{g\to\infty} H_*(\mathfrak{g}_g),\quad
H_*(\mathfrak{g}^+_\infty)=\lim_{g\to\infty} H_*(\mathfrak{g}^+_g).
$$
Since it is well known that the $\mathrm{Sp}$-irreducible decompositions
of $\mathfrak{g}_g(k)$ stabilizes as $g$ goes to $\infty$, we can apply the preceding argument
to conclude that
$$
H_*(\mathfrak{g}_\infty)\cong H_*(\mathfrak{sp}(\infty,\Q))\otimes H_*(\mathfrak{g}^+_\infty)^{\mathrm{Sp}},\quad H_*(\mathfrak{g}^+_\infty)^{\mathrm{Sp}}=
\bigoplus_{w=1}^\infty H_*(\mathfrak{g}^+_\infty)^{\mathrm{Sp}}_w.
$$
Similar formulas are also valid for the continuous cohomology,
although we have to be careful here because the projective limit
arises rather than the direct limit.

Since $H_*(\mathfrak{g}^+_\infty)^{\mathrm{Sp}}_w$ is finite dimensional by the assumption,
we have its Euler characteristic $\chi(H_*(\mathfrak{g}^+_\infty)^{\mathrm{Sp}}_w)$.
We call
$$
g(t)=1+\sum_{w=1}^\infty \chi(H_*(\mathfrak{g}^+_\infty)^{\mathrm{Sp}}_w) \ t^w
\in \Z[[t]]
$$
the {\it weight} generating function for the $\mathrm{Sp}$-invariant stable homology
group of the $\mathrm{Sp}$-Lie algebra $\mathfrak{g}^+$. Observe here that
if we replace $\mathfrak{g}^+$ with $\mathfrak{g}$ here, then we obtain the trivial
function $1$ because $\chi(H_*(\mathfrak{sp}(2g,\Q)))=0$.

\begin{remark}
We mention that such a kind of generating function
was first considered by Perchik \cite{p} in the context of the 
{\it unstable} Gelfand-Fuks cohomology
of the Lie algebra of formal Hamiltonian vector fields
($\mathfrak{ham}_{2g}$ in the notation of \cite{kontsevich3} and Section $5$ below)
and later in \cite{km} in the context of the
{\it stable} Gelfand-Fuks cohomology of 
$\lim_{g\to\infty}\mathfrak{ham}_{2g}^0=\widehat{\mathfrak{c}}_{\infty}\otimes\R$.
\end{remark}

In the case where $\mathfrak{g}_g$ is one of the three Lie algebras considered in
this paper, we can consider the {\it weight} generating function for the 
$\mathrm{Sp}$-invariant homology group of the
limit Lie algebra $\mathfrak{g}^+_\infty$. 

For later use, we generalize the definition of the {\it weight} generating function
in a broader context as follows.

Let 
$$
\mathcal{K}=\bigoplus_{d,w=0}^\infty  \mathcal{K}_{d,w}
$$
be a bigraded algebra over $\Q$ such that the multiplication 
$$
\mathcal{K}_{d,w}\otimes \mathcal{K}_{d',w'}\longrightarrow \mathcal{K}_{d+d',w+w'}
$$
is graded commutative with respect to $d$ (called the degree) and the weight $w$
part
$$
\mathcal{K}_w=\bigoplus_d \mathcal{K}_{d,w} 
$$
is finite dimensional for any $w$. We define the Euler characteristic 
$\chi(\mathcal{K}_w)$ by
$$
\chi(\mathcal{K}_w)=\sum_{d} (-1)^d \mathrm{dim}\ \mathcal{K}_{d,w}.
$$
We also assume that $\mathcal{K}_{0,0}=\Q$. Hereafter we always assume
the above conditions 
whenever we mention 
weight generating functions of bigraded algebras.
It is easy to see that if there are given two bigraded algebras
$\mathcal{K},\mathcal{K}'$ which satisfy the above conditions, then
the tensor product $\mathcal{K}\otimes\mathcal{K}'$ also
satisfies them with respect to the induced
bigradings on it.

\begin{definition}
We define the 
{\it weight} generating function $k(t)$ for a bigraded algebra $\mathcal{K}$ as above
to be
$$
k(t)=1+\sum_{w=1}^\infty \chi(\mathcal{K}_w)\ t^w.
$$
\end{definition}

\begin{example}
Let $P=\Q[x_1,x_2,\ldots]$ be the polynomial algebra on given variables
$x_i$ with degree $2d_i$. If we set the weight
to be equal to the degree, then the weight generating function for $P$ is given by
$$
p(t)=\prod_{i} (1+t^{2d_i}+t^{4d_i}+\cdots)
=\prod_{i}(1-t^{2d_i})^{-1}.
$$
Let $E=\wedge^*[y_1,y_2,\ldots]$ be the exterior algebra on given variables $y_i$ with degree
$2s_i-1$. 
If we set the weight to be equal to the degree, then the weight generating function for $E$
is given by
$$
e(t)=\prod_i (1-t^{2s_i-1}).
$$
\label{ex:w}
\end{example}

\begin{prop}
Let $\mathcal{K}$ be a bigraded algebra and assume that it is
a free graded commutative algebra with respect to the grading by degrees.
Let $P\mathcal{K}$ be the subspace consisting of primitive elements.
For each $d, w$, set $P\mathcal{K}_{d,w}=P\mathcal{K}\cap \mathcal{K}_{d,w}$
and define
$$
\chi^{\mathrm{pr}}_w(\mathcal{K})=\sum_{d} (-1)^d \mathrm{dim}\ P\mathcal{K}_{d,w}.
$$
Then we have
$$
k(t)=\prod_{w=1}^\infty (1-t^w)^{-\chi^{\mathrm{pr}}_w(\mathcal{K})}.
$$
\label{prop:wp}
\end{prop}

\begin{proof}
It is easy to see that the weight generating function of the tensor product of two
graded commutative bialgebras is the product of those of each bigraded algebra.
Since $\mathcal{K}$ is free by the assumption, it is the tensor product of 
subalgebras generated by $P\mathcal{K}$. Consider the product of weight generating functions 
of two graded commutative algebras each of which is generated by an element whose 
weight is the same whereas the
degree is complementary, namely one is even and the other is odd.
Then by Example \ref{ex:w}, we see that this product is the constant function $1$.
Hence the weight generating function of $\mathcal{K}$ depends only on 
$\chi^{\mathrm{pr}}_w(\mathcal{K})$ and the claim follows.
\end{proof}

\begin{remark}
The number $\chi^{\mathrm{pr}}_w(\mathcal{K})$
can be interpreted as the Euler characteristic of the space of primitive
elements of weight $w$, or equivalently as the number of
new generators of weight $w$ with {\it even} degrees minus that of 
new generators of weight $w$ with {\it odd} degrees.
\end{remark}

\begin{prop}
Let $\mathcal{K}$ be a bigraded algebra and assume that it is
a free graded commutative algebra with respect to the grading by degrees.
Assume that the weight generating function $k(t)$ is determined up to
weight $w_0$. Then we can
determine the numbers 
$\chi^{\mathrm{pr}}_w(\mathcal{K})\ (w=1,2,\ldots,w_0)$
inductively by the following recursive formula
$$
\chi^{\mathrm{pr}}_w(\mathcal{K})=[k(t)]_{t^w}-
\left[\prod_{i=1}^{w-1} (1-t^i)^{-\chi^{\mathrm{pr}}_{i}(\mathcal{K})}\right]_{t^w}
\quad (w=1,2,\ldots,w_0)
$$
where $[f(t)]_{t^w}$ denotes the coefficient of $t^w$ in a given formal power
series $f(t)\in \Z[[t]]$.
\label{prop:wc}
\end{prop}

\begin{proof}
This follows from Proposition \ref{prop:wp}.
\end{proof}

\section{Computation of the irreducible decompositions}\label{sec:irrep}

In this section we describe our explicit determination of 
the stable irreducible decompositions of 
$\mathfrak{h}_{g,1}(k), \mathfrak{a}_g(k)$ up to certain degrees. 
At present, we have determined them for all $k\leq 20$.
As already mentioned
in the introduction, in the commutative case,  we have an isomorphism
$\mathfrak{c}_g(k)\cong S^{k+2} H_\Q$ which is known to be an irreducible representation
for any $k$.

To describe our result, we fix our notations.
Any Young diagram $\la=[\la_1\cdots\la_h]$ with
$k$ boxes defines an irreducible representation 
of the symmetric group $\mathfrak{S}_k$ which we denote
by $\la_{\mathfrak{S}_k}$ or sometimes simply by the same symbol $\la$. 
Thus $[k]$ corresponds to the trivial representation
and $[1^k]$ the alternating representation. Here we use
a simplified notation to express Young diagrams. For example
$[422]$ will be denoted by $[42^2]$.
For any Young diagram $\la=[\la_1\cdots\la_h]$ as above and for any
$n\geq h$, let 
$
\la_{\mathrm{GL}}
$
be the corresponding irreducible representation of $\mathrm{GL}(n,\Q)$.
Similarly for any $g\geq h$, let 
$
\la_{\mathrm{Sp}}
$
be the corresponding irreducible representation of $\mathrm{Sp}(2g,\Q)$.

Our method of computing the stable irreducible decompositions for the 
Lie case $\mathfrak{h}_{g,1}$ as well as the associative case
$\mathfrak{a}_g$ can be described as follows.
For the former case, we use the following result of Kontsevich. 

\begin{thm}[Kontsevich \cite{kontsevich1}\cite{kontsevich2}]\label{thm:k}
Let $W_k$ be the $\mathfrak{S}_{k+2}$-module with character
$\chi_k(1^{k+2})=k!, \chi_k(1^1a^b)=(b-1)! a^{b-1} \mu(a),
\chi_k(a^b)=-(b-1)! a^{b-1} \mu(a)$, and $\chi_k$ vanishes
on all the other
conjugacy classes, where $\mu$ denotes the M\"obius
function. Then there exists an isomorphism
$$
\mathfrak{h}_{g,1}(k)\cong H_\Q^{\otimes (k+2)}\otimes_{A_{k+2}} W_k
$$
of\/ $\mathrm{Sp}(2g,\Q)$-modules, where $A_{k+2}=\Q\mathfrak{S}_{k+2}$
denotes the group algebra of $\mathfrak{S}_{k+2}$.
\end{thm}

We have given in \cite{mss2} a simple proof of this result using only the standard representation theory (see also \cite{es}). 
As an immediate corollary to this theorem, we obtain the following.

\begin{cor}\label{cor:h}
Let $\la=[\la_1\cdots\la_{h}]$ be a Young diagram
with $(k+2)$ boxes and let $\la_{\mathfrak{S}_{k+2}}$ be the corresponding
irreducible representation of the symmetric group
$\mathfrak{S}_{k+2}$. 
Then the multiplicity $m_\la$ of
$V_\la=H_\Q^{\otimes (k+2)}\otimes_{A_{k+2}}\la_{\mathfrak{S}_{k+2}}$ in $\mathfrak{h}_{g,1}(k)$ is expressed as
$$
m_\la=\frac{1}{(k+2)!} \sum_{\gamma\in \mathfrak{S}_{k+2}}
\chi_k (\gamma)\chi_\la(\gamma)
$$
where $\chi_\la$ denotes the character of $\la_{\mathfrak{S}_{k+2}}$.
\end{cor}

For the case of $\mathfrak{a}_g(k)$, we have the following result.

\begin{prop}\label{prop:a}
Let $\la=[\la_1\cdots\la_{h}]$ be a Young diagram
with $(k+2)$ boxes and let $\la_{\mathfrak{S}_{k+2}}$ be the corresponding
irreducible representation of the symmetric group
$\mathfrak{S}_{k+2}$. 
Then the multiplicity $n_\la$ of
$V_\la=H_\Q^{\otimes (k+2)}\otimes_{A_{k+2}}\la_{\mathfrak{S}_{k+2}}$ in $\mathfrak{a}_{g}(k)$ is expressed as
$$
n_\la=\frac{1}{k+2} \sum_{i=1}^{k+2}
\chi_\la(\sigma_{k+2}^i)
$$
where $\sigma_{k+2}\in\mathfrak{S}_{k+2}$ denotes the cyclic permutation
$(12\cdots (k+2))$ of order $k+2$.
\end{prop}

\begin{proof}
As is well known, for any $k$ we have an isomorphism 
$$
H_\Q^{\otimes (k+2)}\cong \bigoplus_{|\la|=k+2}\ (\mathrm{dim}\, \la_{\mathfrak{S}_{k+2}}) V_\la
$$
as $\mathrm{GL}(2g,\Q)$-modules, where $|\la|$ denotes 
the number of the boxes of the Young diagram $\la$.
On the other hand, we have an isomorphism
$$
\mathfrak{a}_g(k)=\left(H_\Q^{\otimes (k+2)}\right)^{\Z/(k+2)}
$$
where the cyclic group $\Z/(k+2)$ of order $k+2$ acts on
$H_\Q^{\otimes (k+2)}$ by cyclic
permutations. Then the claim follows by considering the restriction
to the subgroup $\Z/(k+2)\subset \mathfrak{S}_{k+2}$ and
applying the standard argument.
\end{proof}

Our explicit irreducible decompositions of $\mathfrak{h}_{g,1}(k)$ and $\mathfrak{a}_g(k)$
as $\mathrm{Sp}(2g,\Q)$-modules are done as follows. First we determine the
irreducible decompositions of these modules as $\mathrm{GL}(2g,\Q)$-modules. 
For this, we use Corollary \ref{cor:h} and Proposition \ref{prop:a}, respectively,
to compute the multiplicities $m_\la, n_\la$ by applying the
formula of Frobenius which expresses the value $\chi_\la(\gamma)$ 
for any given element $\gamma\in\mathfrak{S}_{k+2}$ as the coefficient
of a certain polynomial $f_\la$ with respect to a certain monomial 
$x_\la$ (see e.g. \cite{fh}, Frobenius Formula 4.10). 
We made a systematic computer computation by using this formula.
Recall here that we must take the number of variables for the polynomial $f_\la$
at least as large as the number $h(\la)$ of rows of the Young diagram $\la$.
Hence, the necessary data will get larger and larger as the number $h(\la)$
increases. To overcome this difficulty, we adopted the following
simple argument. Let $\la'$ denote the {\it conjugate} Young diagram
of any given one $\la$. Then, as is well known, we have an isomorphism
$$
\la'_{\mathfrak{S}_{k+2}}\cong \la_{\mathfrak{S}_{k+2}}\otimes [1^{k+2}]
$$
where $[1^{k+2}]$ denotes the $1$-dimensional alternating
representation. Since, for any given $\mathfrak{S}_{k+2}$-module $W$,
its multiplicity of an irreducible representation
$\la$ is equal to that of $\la'$ in the
conjugate representation $W'=W\otimes [1^{k+2}]$, we can easily 
deduce the following.

\begin{prop} 
In the same situation as in Corollary \ref{cor:h} and Proposition \ref{prop:a},
we have
$$
m_\la=\frac{1}{(k+2)!} \sum_{\gamma\in \mathfrak{S}_{k+2}}
(\mathrm{sgn}\;\gamma)\;\chi_k (\gamma)\chi_{\la'}(\gamma)
$$
$$
n_\la=\frac{1}{k+2} \sum_{i=1}^{k+2}\, (\mathrm{sgn}\;\sigma_{k+2}^i)\,
\chi_{\la'}(\sigma_{k+2}^i)
$$
where $\mathrm{sgn}\;\gamma$ denotes the sign of $\gamma$.
\end{prop}

Since the two numbers $h(\la')$ and $h(\la)$ are
so to speak {\it complementary} to each other (e.g.
$h([1^k]')=1$ while $h([1^k])=k$), we can make computer computations
roughly twice as much compared to the situation where we do not use
this method. Also certain symmetries in the structure of $\mathfrak{h}_{g,1}$
as well as $\mathfrak{a}_g$ which we found in \cite{mss2}
decrease necessary computations considerably.

Next we use the known formula of decomposing a given irreducible
$\mathrm{GL}(2g,\Q)$-module $\la_{GL}$ into $\mathrm{Sp}(2g,\Q)$-irreducible
components (see e.g. formula (25.39) in \cite{fh}). We made 
a computer program of this formula and by using it we made a database
which contains the $\mathrm{Sp}(2g,\Q)$-irreducible decomposition of
all the $\mathrm{GL}(2g,\Q)$-modules $\la_{GL}$ with the number
$|\la|$ of boxes of the Young diagram $\la$ less than or equal to 
 $30$
(there are $28628$ such Young diagrams).

As mentioned already, we have so far determined the 
$\mathrm{Sp}(2g,\Q)$-irreducible decompositions of
$\mathfrak{h}_{g,1}(k)$ and $\mathfrak{a}_{g}(k)$ for all $k\leq 20$
by making use of the above method. Although here we omit the description of the 
results, see Tables \ref{tab:i1} and \ref{tab:i2} in which we
express the dimensions of the $\mathrm{Sp}$-invariant subspaces
$\mathfrak{h}_{g,1}(k)^{\mathrm{Sp}}$ for all $k\leq 20$.
These tables contain more precise information on these subspaces.
Namely they contain a complete description how these subspaces
degenerate according as the genus $g$ decreases from the stable
range one by one to the final case $g=1$.

\section{Computation of the dimensions of the $\mathrm{Sp}$-invariant
subspaces of various $\mathrm{Sp}$-modules}\label{sec:di}

In this section, we describe several methods which we developed
in our computer computations. We have to determine the dimensions
of the subspaces consisting of the $\mathrm{Sp}$-invariant
elements of various $\mathrm{Sp}$-modules such as
$$
\left(\wedge^{d_1} \mathfrak{c}_g(i_1)\right)\otimes\cdots\otimes 
\left(\wedge^{d_s} \mathfrak{c}_g(i_s)\right)
$$
or the corresponding modules where we replace $\mathfrak{c}_g$
by $\mathfrak{h}_{g,1}$ or $\mathfrak{a}_g$. We mention that the
character of any of these modules is known so that {\it theoretically}
there is no problem. More precisely, we can adopt the method 
given in \cite{km}, which treated the case of $\mathfrak{c}_g$
by extending the original one due to Perchik \cite{p},
to the other two case as well to obtain closed formulas for the above
dimensions of $\mathrm{Sp}$-invariants. Unfortunately however, 
these formulas are too complicated so that when we use computers
to obtain explicit values, the memory problem arises in a very early stage.
We have developed other methods described as follows.
\par
\vspace{5mm}
Method (I) ($\mathrm{GL}$-decomposition of tensor products)
\par
\vspace{3mm}
There is a formula, called the Littlewood-Richardson rule, which gives
the irreducible decomposition of the tensor product of any two 
$\mathrm{GL}$-modules and a similar formula is known for the case of 
$\mathrm{Sp}$-modules (see e.g. \cite{fh}). However, the latter formula 
is considerably more complicated than the former one. In view of this,
we postpone the $\mathrm{Sp}$-irreducible decomposition as late as possible and we
make the $\mathrm{GL}$-irreducible decomposition as far as possible.
We made a computer program for the Littlewood-Richardson rule and apply
it in various stages in our computation.
\par
\vspace{5mm}
Method (II) ($\mathrm{Sp}$-decomposition of $\mathrm{GL}$-modules)
\par
\vspace{3mm}
There is a combinatorial formula which gives the 
$\mathrm{Sp}$-irreducible decomposition of any irreducible
$\mathrm{GL}$-module $\la_{\mathrm{GL}}$, namely the restriction law corresponding to
the pair $\mathrm{Sp}(2g,\Q)\subset \mathrm{GL}(2g,\Q)$ (see \cite{fh}).
We made a computer program for this procedure and apply
it in various stages in our work.
\par
\vspace{5mm}
Method (III) (Counting the number of Young diagrams with {\it multiple double floors})
\par
\vspace{3mm}
Let us call a Young diagram $\la$ {\it with multiple double floors} if it
has the form $\la=[\la_1\la_1\cdots\la_s\la_s]$. It is easy to see
that $\la$ is such a Young diagram if and only if its conjugate Young
diagram $\la'$ is of {\it even type} in the sense that all the numbers
appearing in it are even integers. Now at the final stage of counting the
dimension of the $\mathrm{Sp}$-invariant subspace of a $\mathrm{GL}$-module
$V$, we can determine the required number without performing the $\mathrm{Sp}$-decomposition
of $V$ by adopting the following method.

\begin{prop}
Let $V$ be a $\mathrm{GL}(2g,\Q)$-module and let $V^{\mathrm{Sp}}$ denote
the subspace consisting of $\mathrm{Sp}(2g,Q)$-invariant elements of $V$
considered as an $\mathrm{Sp}(2g,\Q)$-module. Also let
$$
V=\bigoplus_{\la} m_\la(V)\ \la_{\mathrm{Sp}}
$$
be the $\mathrm{Sp}$-irreducible decomposition of $V$. Then we have the equality
$$
\mathrm{dim}\, V^{\mathrm{Sp}}
=\sum_{\la: \text{multiple double floors}} m_\la(V).
$$
\end{prop}
\begin{proof}
This follows from the fact that 
$$
\mathrm{dim}\, (\la_{\mathrm{GL}})^{\mathrm{Sp}}=
\begin{cases}
&1\quad (\text{$\la$: multiple double floors})\\
&0\quad (\text{otherwise})
\end{cases}
$$
which follows from the restriction law corresponding to
the pair $\mathrm{Sp}(2g,\Q)\subset \mathrm{GL}(2g,\Q)$.
\end{proof}

We made a computer program which counts the number of 
Young diagrams with multiple double floors
in any linear combination of Young diagrams.
\par
\vspace{5mm}
Method (IV) (Counting pairs of Young diagrams with the same shape)
\par
\vspace{3mm}
The difficulty in applying our program of performing the
Littlewood-Richardson rule for the tensor product
$V_1\otimes V_2$
increases according to the numbers of boxes of the Young diagrams
appearing in the irreducible decompositions of $V_i$ get larger and larger. 
In case we cannot obtain the result within an appropriate time, we
adopt this method which depends on the following fact.
\begin{prop}
Let $V_1,V_2$ be two 
$\mathrm{GL}(2g,\Q)$-modules and let
$$
V_i=\bigoplus_{\la} m_\la(V_i)\ \la_{\mathrm{Sp}} \quad (i=1,2)
$$
be the $\mathrm{Sp}$-irreducible decompositions of $V_i$. Then we have the equality
$$
\mathrm{dim}\, (V_1\otimes V_2)^{\mathrm{Sp}}
=\sum_\la  m_\la(V_1) m_\la(V_2).
$$
\end{prop}
\begin{proof}
This follows from the well known fact that
$$
\mathrm{dim}\, (\la_{\mathrm{Sp}}\otimes \mu_{\mathrm{Sp}})^{\mathrm{Sp}}=
\begin{cases}
&1\quad (\la=\mu)\\
&0\quad (\text{otherwise})
\end{cases}
.
$$
\end{proof}
Here is another similar formula.
\begin{prop}
Let $V$ be an
$\mathrm{Sp}(2g,\Q)$-module and let
$$
V=\bigoplus_{\la} m_\la(V)\ \la_{\mathrm{Sp}} 
$$
be its $\mathrm{Sp}$-irreducible decomposition. Then we have the equality
$$
\mathrm{dim}\, (\wedge^2 V)^{\mathrm{Sp}}
=\frac{1}{2}\left(\sum_{\text{$|\la|$: odd}}  m_\la(V) (m_\la(V)+1)+
\sum_{\text{$|\la|$: even}}  (m_\la(V)-1) m_\la(V) \right).
$$
\end{prop}
\begin{proof}
This follows from the well known fact that
$$
\mathrm{dim}\, (\wedge^2\,\la_{\mathrm{Sp}})^{\mathrm{Sp}}=
\begin{cases}
&1\quad (\text{$|\la|$: odd})\\
&0\quad (\text{$|\la|$: even})
\end{cases}
.
$$
\end{proof}

We made a computer program which counts the number of 
pairs with the same Young diagrams in any two linear
combination of Young diagrams. We can use this method to
check the accuracy of our computations by applying it
to plural expressions 
$$
V=V_1\otimes V_2=V'_1\otimes V'_2
$$
as tensor products of the same $\mathrm{GL}(2g,\Q)$-module $V$.

\par
\vspace{5mm}
Method (V) (Adams operations)
\par
\vspace{3mm}
The most difficult part in our computation is the
determination of the $\mathrm{GL}$ as well as
$\mathrm{Sp}$ irreducible decomposition of the 
exterior powers $\wedge^k [1^3]_{\mathrm{GL}}$
of $\mathfrak{h}_{g,1}(1)=[1^3]_{\mathrm{GL}}$.

As is well known, the character of 
$\wedge^k [1]_{\mathrm{GL}}=[1^k]_{\mathrm{GL}}$
is given by
$$
\mathrm{ch}([1^k]_{\mathrm{GL}})=E_k(x_1,...,x_N)
$$
where $E_k$ denotes the $k$-th elementary symmetric 
polynomial and $N$ denotes some fixed large number.
In particular $\mathrm{ch}([1^3]_{\mathrm{GL}})=E_3$.
Then the character of $\wedge^k [1^3]_{\mathrm{GL}}$ can be written
as
$$
\mathrm{ch}(\wedge^k [1^3]_{\mathrm{GL}})
=E_k(E_3)
$$
where $E_k(E_3)$ denotes the $k$-th elementary symmetric polynomial
with respect to the new variables $\{x_ix_jx_k\}_{i<j<k}$.
Here we apply the well-known classical combinatorial algorithm
to express any symmetric polynomial as a polynomial on the
elementary symmetric polynomial to obtain a formula for 
the character of $\wedge^k [1^3]_{\mathrm{GL}}$. For example
$\mathrm{ch}(\wedge^2 [1^3]_{\mathrm{GL}})
=E_6+E_2E_4-E_1E_5$. Then we apply Method (I) to obtain
the $\mathrm{GL}$-irreducible decomposition and further
apply Method (II) to obtain the $\mathrm{Sp}$-irreducible decomposition.
For example we have
\begin{align*}
\wedge^2 [1^3]_{\mathrm{GL}}&=[1^6]_{\mathrm{GL}}+[2^21^2]_{\mathrm{GL}}\\
&=[1^6]_{\mathrm{Sp}}+[2^21^2]_{\mathrm{Sp}}+2 [1^4]_{\mathrm{Sp}}+[21^2]_{\mathrm{Sp}}+[2^2]_{\mathrm{Sp}}+3 [1^2]_{\mathrm{Sp}}+2 [0]_{\mathrm{Sp}}
\end{align*}
where $[0]_{\mathrm{Sp}}$ denotes the trivial representation.
For large $k$, we used the computer software LiE to obtain the
irreducible decompositions. However, because of the memory problem 
we could obtain the $\mathrm{GL}$-irreducible decomposition of 
$\wedge^k [1^3]_{\mathrm{GL}}$ only up to $k=6$ or so.
To overcome this difficulty, we used the Adams operations 
$\psi^k\ (k=1,2,\ldots)$ which satisfy the identity
$$
\wedge^k V=\frac{1}{k}\left(\wedge^{k-1}V\otimes V-\wedge^{k-2}V\otimes \psi^2(V)+\cdots +(-1)^{k-1}\psi^k(V)\right)
$$
on any representation $V$. It turns out that the computer computation
of the Adams operation is much easier than that of the exterior powers.
By utilizing this merit of the Adams operations, we have determined so far the $\mathrm{GL}$-irreducible decomposition of 
$\wedge^k [1^3]_{\mathrm{GL}}$ for $k\leq 10$ and
the $\mathrm{Sp}$-irreducible decomposition of 
$\wedge^k [1^3]_{\mathrm{GL}}$ for $k\leq 9$. As for the 
computation of the dimensions of the $\mathrm{Sp}$-invariant
subspaces of $\mathrm{GL}$-modules with the form
$\wedge^k [1^3]_{\mathrm{GL}}\otimes V$ for certain $V$, we can
go further up to $k=16$ or so. Here we express
$\wedge^k [1^3]_{\mathrm{GL}} (k=10, 11,\ldots)$ in terms of
$\wedge^k [1^3]_{\mathrm{GL}} (k=1, 2, \ldots ,9)$ and 
$\psi^k [1^3]_{\mathrm{GL}} (k=1,2,\ldots)$ and apply the preceding methods.
The coefficients of this expression are complicated
rational numbers rather than the integers.
From this fact, we obtain an {\it extra merit}
of this method. Namely, we can check the accuracy of the computation
just by confirming the answer to be an {\it integer} because
it is most likely that any small mistake in the computation
would force that the output is {\it not} an integer.

\par
\vspace{5mm}
Method (VI) (Counting the number of graphs with a prescribed type)
\par
\vspace{3mm}

In \cite{morita96}, a certain linear mapping
$$
\Q\langle\text{isomorphism class of trivalent graph with $2k$-verticies}\rangle
\rightarrow
\left(\wedge^{2k}[1^3]_{\mathrm{GL}}\right)^{\mathrm{Sp}}
$$
was introduced by making use of a classical result of Weyl,   which is an isomorphism in the stable range.
Here the left hand side denotes the vector space generated by the isomorphism classes of trivalent graphs with $2k$ vertices
where we allow a trivalent graph to have multi-edges and/or
loops. In the theory of enumeration of graphs, the numbers of such trivalent graphs are known for $k\leq 16$ 
by making use of the result of Read \cite{read}.

\par
\vspace{5mm}
Method (VII) (Checking the accuracy of computations)
\par
\vspace{3mm}

We have adopted a few checking procedure to confirm the accuracy of our
computations. As for the irreducible decompositions, we have checked that
the dimension of the resulting decomposition coincides with that of the original
module by applying the Weyl character formula. As for the dimension counting
of various $\mathrm{Sp}$-invariant subspaces, we carried out multiple different ways
of computations and checked that the answers coincide with each other.

In short, our strategy is a mixture of theoretical considerations and computer computations.
By combining the above Methods (I)-(VII) in various ways, 
we made explicit computer computations the results of which will be given in the following three
sections.

\section{The case of $\mathfrak{c}_g$ and the graph homology as well as transversely symplectic foliations}\label{sec:c}

First we consider the commutative case.
From the point of view of explicit computations,
this case of $\mathfrak{c}_g$ is the simplest among the three
Lie algebras because each piece $\mathfrak{c}_g(k)\cong S^{k+2} H_\Q$
is a single irreducible $\mathrm{Sp}(2g,\Q)$-module. 
However, its stable (co)homology is
far from being well understood and there are big mysteries here.
Before describing them, the result of our computation for this case
is depicted in Table \ref{tab:c}.

\begin{table}[h]
\caption{$\text{Case of $\mathfrak{c}_g$}$}
\begin{center}
\begin{tabular}{|c|r|r|r|r|r|r|r|r|r|r|}
\noalign{\hrule height0.8pt}
\hline
$w$ & $2$ & $4$ & $6$ & $8$ & $10$ & $12$ & $14$ & $16$ & $18$ & $20$\\
\hline
$C_1$ & $0$ & $0$ & $0$ & $0$ & $0$ & $0$ & $0$ & $0$ & $0$ & $0$\\
\hline
$C_2$ & $1$ & $0$ & $1$ & $0$ & $1$ & $0$ & $1$ & $0$ & $1$ & $0$\\
\hline
$C_3$ & {} & $1$ & $1$ & $4$ & $3$ & $8$ & $6$ & $12$ & $10$ & $17$\\
\hline
$C_4$ & {} & $3$ & $0$ & $16$ & $20$ & $63$ & $78$ & $164$ & $205$ & $355$\\
\hline
$C_5$ & {} & {} & $4$ & $20$ & $112$ & $271$ & $748$ & $1484$ & $3103$ & $5447$\\
\hline
$C_6$ & {} & {} & $7$ & $15$ & $269$ & $1013$ & $3964$ & $11047$ & $29423$ & $67611$\\
\hline
$C_7$ & {} & {} & {} & $25$ & $310$ & $2784$ & $14034$ & $59153$ & $200982$ & $613281$\\
\hline
$C_8$ & {} & {} & {} & $24$ & $223$ & $4690$ & $36530$ & $220693$ & $1023318$ & $4068707$\\
\hline
$C_9$ & {} & {} & {} & {} & $166$ & $4683$ & $68504$ & $592111$ & $3862954$ & $20226716$\\
\hline
$C_{10}$ & {} & {} & {} & {} & $86$ & $2963$ & $87552$ & $1167459$ & $10828229$ & $76399055$\\
\hline
$C_{11}$ & {} & {} & {} & {} & {} & $1395$ & $73358$ & $1682134$ & $22709573$ & $220634704$\\
\hline
$C_{12}$ & {} & {} & {} & {} & {} & $426$ & $39797$ & $1727415$ & $35748802$ & $488935936$\\
\hline
$C_{13}$ & {} & {} & {} & {} & {} & {} & $13984$ & $1221607$ & $41935536$ & $832479480$\\
\hline
$C_{14}$ & {} & {} & {} & {} & {} & {} & $2732$ & $570419$ & $35952084$ & $1085617203$\\
\hline
$C_{15}$ & {} & {} & {} & {} & {} & {} & {} & $164365$ & $21796235$ & $1073488879$\\
\hline
$C_{16}$ & {} & {} & {} & {} & {} & {} & {} & $23701$ & $8867266$ & $789223120$\\
\hline
$C_{17}$ & {} & {} & {} & {} & {} & {} & {} & {} & $2199842$ & $417233525$\\
\hline
$C_{18}$ & {} & {} & {} & {} & {} & {} & {} & {} & $258951$ & $149905889$\\
\hline
$C_{19}$ & {} & {} & {} & {} & {} & {} & {} & {} &  {}  & $32900910$\\
\hline
$C_{20}$ & {} & {} & {} & {} & {} & {} & {} & {} &  {}  & $3365151$\\
\hline
$\text{total}$ & $1$ & $4$ & $13$ & $104$ & $1190$ & $18296$  & $341288$ & $7441764$ & $185416514$ & $5195165986$\\
\hline
$\chi$ & $1$ & $2$ & $3$ & $6$ & $8$ & $14$ & $20$ & $32$  & $44$ & $68$\\
\noalign{\hrule height0.8pt}
\end{tabular}
\end{center}
\label{tab:c}
\end{table}

Here $C_k$ of the weight $w$ part denotes
$$
\lim_{g\to\infty}\
\mathrm{dim}
\left(
\bigoplus_{\begin{subarray}{c}
i_1+ i_2+\cdots + i_w=k\\ 
i_1+2 i_2+\cdots +w i_w=w
\end{subarray}}
\wedge^{i_1} (S^{3} H_\Q)\otimes\wedge^{i_2} (S^{4} H_\Q)\otimes\cdots\otimes
\wedge^{i_w} (S^{w+2} H_\Q)\right)^{\mathrm{Sp}}
$$
so that we have a {\it finite} dimensional chain complex
$0\to C_{w}\to\cdots\to C_{1}\to 0$.
Also $\chi$ denotes the Euler characteristic of this chain complex,
namely the weight $w$ part of the
$\mathrm{Sp}$-invariant stable homology
$
H_*(\mathfrak{c}^+_\infty)^{\mathrm{Sp}}_w.
$

\begin{proof}[Proof of Theorem $\ref{thm:chi}\ \mathrm{(i)}$]
This follows from Table \ref{tab:c}.
\end{proof}

Let $G_*^{(n)}\ (n\geq 2)$ be the graph complex defined by Kontsevich, which is a chain complex of
dimension $2n-2$.

\begin{thm}[Kontsevich {\cite{kontsevich1, kontsevich2}}]\label{thm:kontsevich-c}
For any $k \ge 1$ and $n \ge 1$, there exists an isomorphism 
\[PH_k (\mathfrak{c}^+_{\infty})^{\mathrm{Sp}}_{2n}
\cong 
H_{k} (G_*^{(n+1)}).\]
\label{thm:kc}
\end{thm}

\begin{prop}
The weight generating function, denoted by $c(t)$, for the
$\mathrm{Sp}$-invariant stable homology group $H_*(\mathfrak{c}^+_{\infty})^{\mathrm{Sp}}$
is given by
$$
c(t)=\prod_{n=2}^\infty (1-t^{2n-2})^{- e(G_*^{(n)})}
$$
where $e(G_*^{(n)})$ denotes the Euler characteristic of $H_*(G_*^{(n)})$.
\label{prop:ct}
\end{prop}
\begin{proof}
This follows from Theorem \ref{thm:kc} and Proposition \ref{prop:wp}.
\end{proof}

\begin{proof}[Proof of Theorem $\ref{thm:chip}\ \mathrm{(i)}$]
By Theorem \ref{thm:chi} (i), we see that the weight generating function
$$
c(t)=\sum_{w=0}^\infty \chi(H_*(\mathfrak{c}^+_\infty)^{\mathrm{Sp}}_w)t^w
$$
for $H_*(\mathfrak{c}^+_\infty)^{\mathrm{Sp}}$, up to weight $20$, is given by
$$
c(t)=1+t^2+2t^4+3t^6+6t^8+8t^{10}+14t^{12}+20t^{14}+32t^{16}+44t^{18}+68t^{20}+\cdots.
$$
By applying Proposition \ref{prop:wc}, we can inductively determine the integral Euler characteristics of the
primitive parts, namely $e(G_*^{(n)})$. 
If we put
\begin{align*}
\bar{c}(t)=(1-t^2)^{-1}&(1-t^4)^{-1}(1-t^6)^{-1}(1-t^8)^{-2}\\
&(1-t^{10})^{-1}(1-t^{12})^{-2}(1-t^{14})^{-2}(1-t^{16})^{-2}(1-t^{18})^{-1}(1-t^{20})^{-3},
\end{align*}
then we see that
$$
c(t)-\bar{c}(t)\equiv 0\,  
\bmod t^{21}.
$$
By Proposition \ref{prop:ct}, we can now conclude that $e(G_*^{(n)})=1,1,1,2,1,2,2,2,1,3$ for $n=2,3,\ldots,11$, respectively.
The result is depicted in the fourth row of Table \ref{tab:cn}.
\end{proof}

\begin{table}[h]
\caption{$\text{Numbers of new generators for $H_*(\mathfrak{c}^+_\infty)^{\mathrm{Sp}}_w$}$}
\begin{center}
\begin{tabular}{|c|r|r|r|r|r|r|r|r|r|r|}
\noalign{\hrule height0.8pt}
\hline 
$w$ & $2$ & $4$ & $6$ & $8$ & $10$ & $12$ & $14$ & $16$  & $18$  & $20$\\
\hline
$\chi$ & $1$ & $2$ & $3$ & $6$ & $8$ & $14$ & $20$ & $32$  & $44$   & $68$ \\
\hline
$\text{$\chi$ of lower terms}$ & $0$ & $1$ & $2$ & $4$ & $7$ & $12$ & $18$ & $30$  & $43$   & $65$ \\
\hline
$\text{$\chi$ of primitive part}$ & $1$ & $1$ & $1$ & $2$ & $1$ & $2$ & $2$ & $2$  & $1$    & $3$\\
\hline
$\mathrm{dim}\,\mathcal{A}(\emptyset)^{(w)}$ & $1$ & $2$ & $3$ & $6$ & $9$ & $16$ & $25$ & $42$  & $65$   & $105$ \\
\hline
$\text{generators for $\mathcal{A}(\emptyset)$}$ & $1$ & $1$ & $1$ & $2$ & $2$ & $3$ & $4$ & $5$  & $6$   & $8$ \\
\hline
$\text{$\chi$ of primitive part of $\mathcal{E}$}$ & $0$ & $0$ & $0$ & $0$ & $-1$ & $-1$ & $-2$ & $-3$  & $-5$   & $-5$ \\
\noalign{\hrule height0.8pt}
\end{tabular}
\end{center}
\label{tab:cn}
\end{table}

As was already mentioned in the introduction,
there are deep connections of this case with
two important subjects in topology. Namely the theory of {\it finite type} invariants
of homology $3$-spheres as well as $3$-manifolds and the theory of 
characteristic classes of {\it transversely symplectic} foliations.

In \cite{gn} (Theorem 2), Garoufalidis and Nakamura proved the following 
beautiful result. Stably there exists an isomorphism
$$
\mathcal{A}(\emptyset)
\cong \left(\wedge^* (S^3H_\Q)/([4]_{\mathrm{Sp}})\right)^{\mathrm{Sp}}
$$
of graded algebras. Here 
$[4]_{\mathrm{Sp}}=S^4H_\Q\subset \wedge^2 (S^3H_\Q)$ denotes a certain
summand and $([4]_{\mathrm{Sp}})$ denotes the ideal of $\wedge^* (S^3H_\Q)$
generated by it. 
This result is closely related to
the following well-known fact which shows that the top homology groups
of the graph complexes correspond to the finite type
invariants of $3$-manifolds. 
\begin{prop}
There exists an isomorphism
$$
H_{2n-2}(G_*^{(n)})\cong \mathcal{A}(\emptyset)^{(2n-2)}_{\mathrm{conn}}.
$$
\label{prop:ce}
\end{prop}
Here we derive this fact from the above theorem of 
Garoufalidis and Nakamura, in the framework of this paper,
just for completeness.
In the chain complex computing $H_*(\mathfrak{c}^+_g)_2$
(the weight $2$ part), the boundary operator 
is the Poisson bracket
$$
\wedge^2(S^3H_\Q) \overset{\partial}{\longrightarrow} S^4H_\Q
$$
which is easily seen to be {\it surjective}. Passing to the dual, the cochain complex computing $H_c^*(\widehat{\mathfrak{c}}^+_g)_2$
is
$$
S^4H_\Q\overset{\delta}{\longrightarrow} \wedge^2(S^3H_\Q) 
$$
which is easily seen to be {\it injective}. 
Here recall that any finite dimensional $\mathrm{Sp}$-module
is canonically isomorphic to its dual module. 
Since the multiplicity of $S^4 H_\Q$
in $\wedge^2(S^3H_\Q)$ is one as already pointed out in \cite{gn}, this is the
same summand as above.
By the definition of the Lie algebra cohomology, we can now conclude that the final part of the
cochain complex computing $H_c^*(\widehat{\mathfrak{c}}^+_g)^{\mathrm{Sp}}_{2k}$
is
$$
\cdots \rightarrow 
\left(S^4H_\Q\otimes \wedge^{2k-2}(S^3H_\Q)\right)^{\mathrm{Sp}}
\xrightarrow{\wedge (\delta\otimes\mathrm{id})}
\left(\wedge^{2k}(S^3H_\Q)\right)^{\mathrm{Sp}}\rightarrow 0
$$
where the last non-trivial homomorphism can be identified with $\delta\otimes\mathrm{id}$ followed 
by the wedge product $\wedge$. This is because $\delta(\alpha\wedge\beta)=(\delta \alpha)\wedge\beta
+(-1)^{\mathrm{deg}\,\alpha}\alpha\wedge(\delta\beta)$ in general and
$\delta \beta=0$ for any $\beta\in S^3H_\Q$ in the present case.
Now the top cohomology group $H_c^{2k}(\widehat{\mathfrak{c}}^+_g)^{\mathrm{Sp}}_{2k}$ is the cokernel of the above homomorphism
and clearly the image of $\wedge(\delta\otimes\mathrm{id})$ coincides with the ideal
$\left(S^4H_\Q\right)$. One can now apply the result of Garoufalidis and Nakamura above
to obtain an isomorphism
$$
H_c^{2k}(\widehat{\mathfrak{c}}^+_g)^{\mathrm{Sp}}_{2k}\cong \mathcal{A}(\emptyset)^{(2k)}.
$$
By passing to the dual, this yields an isomorphism
$$
H_{2k}(\mathfrak{c}^+_g)^{\mathrm{Sp}}_{2k}\cong \mathcal{A}(\emptyset)^{(2k)}
$$
of the top homology group. Proposition \ref{prop:ce} follows from this
by restricting to the primitive part.

We can now conclude that there exists an injective homomorphism
$$
\mathcal{A}(\emptyset)\rightarrow H_*(\mathfrak{c}^+_\infty)^{\mathrm{Sp}}
$$
and let us consider $\mathcal{A}(\emptyset)$ as a subalgebra of $H_*(\mathfrak{c}^+_\infty)^{\mathrm{Sp}}$.

\begin{definition}
Let $\mathcal{A}^+(\emptyset)$ denote the subalgebra of $\mathcal{A}(\emptyset)$  
consisting of all the elements with {\it positive} degrees
and let $\mathcal{I}(\mathcal{A}^+(\emptyset))$ denote the ideal of $H_*(\mathfrak{c}^+_\infty)^{\mathrm{Sp}}$
generated by $\mathcal{A}^+(\emptyset)$.
Now set
$$
\mathcal{E}=H_*(\mathfrak{c}^+_\infty)^{\mathrm{Sp}}/\mathcal{I}(\mathcal{A}^+(\emptyset))
$$
which is a free graded commutative algebra with respect to the degree. It is also equipped with the
second grading induced by the weights. 
\label{def:e}
\end{definition}

By the definition, clearly we have an isomorphism
$$
H_*(\mathfrak{c}^+_\infty)^{\mathrm{Sp}}\cong \mathcal{A}(\emptyset)\otimes \mathcal{E}
$$
of bigraded algebras.

\begin{proof}[Proof of Theorem $\ref{thm:ce}$]
Although the structure of the polynomial algebra
$
\mathcal{A}(\emptyset)
$
is far from being understood,
it is known that the numbers of generators for this algebra are $1,1,1,2,2,3,4,5,6,8,9$ for degrees $w=2,4,\ldots, 22$ (see \cite{ohe})
and the generating function of this algebra is
$$
\phi(t)=1+t^2+2t^4+3t^6+6t^8+9t^{10}+16t^{12}+25t^{14}+42t^{16}+65t^{18}+105t^{20}+161t^{22}+\cdots.
$$
We write these values in the fifth and the sixth rows of Table \ref{tab:cn}.
Then by subtracting the sixth row from the fourth row of Table \ref{tab:cn},
we can determine the
first several terms of the weight generating function 
for the bigraded algebra $\mathcal{E}$
to be
$$
e(t)=1-t^{10}-t^{12}-2 t^{14}-3t^{16}-5t^{18}-5t^{20}+\cdots.
$$
Of course we should have the identity $c(t)=\phi(t) e(t)$ which is easy to check.
This completes the proof.
\end{proof}

In the framework of our bigraded algebra $\mathcal{E}$, the results of Gerlits \cite{ge} (Theorem 4.1.)
as well as Conant, Gerlits and Vogtmann \cite{cgv} (Theorem 5.1.)
can be described as follows. Namely, $\mathcal{E}_w=0$ for all $w=2,\ldots,8$ and
$\mathcal{E}_{10}\cong \Q, \mathcal{E}_{12}\cong \Q$ 
are spanned by certain elements in 
$PH_7(\mathfrak{c}^+_\infty)^{\mathrm{Sp}}_{10}$ and
$PH_9(\mathfrak{c}^+_\infty)^{\mathrm{Sp}}_{12}$, respectively.

\begin{conj}
The free graded algebra $\mathcal{E}$
is {\it infinitely} generated. Furthermore there exist infinitely many
generators with {\it odd} degrees.
\label{conj:e}
\end{conj}

\begin{problem}
Construct explicit cycles lying in $P\mathcal{E}$.
\end{problem}

Next we describe the connection of the commutative case with the theory of
characteristic classes of transversely symplectic foliations.
Let $\mathfrak{ham}_{2g}$ denote the Lie algebra consisting of all the formal Hamiltonian
vector fields on $\R^{2g}$ with respect to the standard symplectic form. In \cite{kontsevich3},
Kontsevich considered two Lie subalgebras
$$
\mathfrak{ham}^1_{2g}\subset  \mathfrak{ham}^0_{2g}\subset \mathfrak{ham}_{2g}
$$
where $\mathfrak{ham}^0_{2g}$ and  $\mathfrak{ham}^1_{2g}$ denote the Lie subalgebra consisting of formal Hamiltonian vector fields without constant terms and  without constant as well as
linear terms, respectively. He gave a geometric meaning to the Gelfand-Fuks cohomology
$$
H_{GF}^*(\mathfrak{ham}^0_{2g},\mathrm{Sp}(2g,\R))\cong 
H^*_{GF}(\mathfrak{ham}^1_{2g})^{\mathrm{Sp}}
$$
as follows. Let $\mathcal{F}$ be a transversely symplectic foliation on a smooth manifold
$M$ of codimension $2g$ and let $H^*_{\mathcal{F}}(M)$ be the associated foliated cohomology group.
Then he constructed a homomorphism
$$
H^*_{GF}(\mathfrak{ham}^1_{2g})^{\mathrm{Sp}}\rightarrow H^*_{\mathcal{F}}(M).
$$
Now it is easy to see that the Lie algebras $\mathfrak{ham}^0_{2g}, \mathfrak{ham}^1_{2g}$
are nothing other than the completions of $\mathfrak{c}_g\otimes\R, \mathfrak{c}^+_g\otimes\R$
with respect to the natural gradings so that
we can write
$$
\mathfrak{ham}^0_{2g}=\widehat{\mathfrak{c}}_g\otimes\R,\quad
\mathfrak{ham}^1_{2g}=\widehat{\mathfrak{c}}^+_g\otimes\R.
$$
It follows that we have a homomorphism
$$
H^*_c(\widehat{\mathfrak{c}}^+_g)^{\mathrm{Sp}}\otimes\R\cong
H^*_{GF}(\mathfrak{ham}^1_{2g})^{\mathrm{Sp}}\rightarrow H^*_{\mathcal{F}}(M)
$$
for any transversely symplectic foliation $(M,\mathcal{F})$. 
Let $\mathfrak{c}^+_g\rightarrow \mathfrak{c}_g(1)=S^3H_\Q$ be the projection.
Then the composition
$$
H^*(S^3H_\Q)^{\mathrm{Sp}}\rightarrow H^*_c(\widehat{\mathfrak{c}}^+_g)^{\mathrm{Sp}}\otimes\R\rightarrow H^*_{\mathcal{F}}(M)
$$
produces the {\it usual} leaf cohomology classes in the sense that they are expressed
by differential forms involving only the connection form and the curvature form including
the Pontrjagin forms. It follows that our bigraded algebra $\mathcal{E}$ can be interpreted as the dual
of the space of all the {\it exotic} stable leaf cohomology classes, as already mentioned
in the introduction.

\begin{problem}
Study the geometric meaning of the classes in $\mathcal{E}$ in the context of
universal characteristic classes for {\it odd} dimensional manifold bundles
as well as characteristic classes for transversely symplectic foliations.
\end{problem}

\section{The case of $\mathfrak{h}_{g,1}$ and the outer automorphism groups of free groups}\label{sec:l}

Next we consider the Lie case.
The result of our computation for this case
is depicted in Table \ref{tab:h}. 

\begin{table}[h]
\caption{$\text{Case of $\mathfrak{h}_{g,1}$}$}
\begin{center}
\begin{tabular}{|c|r|r|r|r|r|r|r|r|r|}
\noalign{\hrule height0.8pt}
\hline
$w$ & $2$ & $4$ & $6$ & $8$ & $10$ & $12$ & $14$ & $16$  & $18$ \\
\hline
$C_1$ & $1$ & $0$ & $5$ & $3$ & $108$ & $650$ & $8817$ & $111148$ & $1729657$\\
\hline
$C_2$ & $2$ & $0$ & $10$ & $66$ & $580$ & $6621$ & $84756$ & $1281253$ & $21671535$\\
\hline
$C_3$ & {} & $6$ & $7$ & $239$ & $1928$ & $29219$ & $424358$ & $7286710$ & $137344661$\\
\hline
$C_4$ & {} & $8$ & $16$ & $342$ & $4946$ & $78443$ & $1400274$ & $27097563$ & $575398310$\\
\hline
$C_5$ & {} & {} & $41$ & $293$ & $8375$ & $152310$ & $3289532$ &$73457788$  & $1766236662$\\
\hline
$C_6$ & {} & {} & $31$ & $287$ & $8887$ & $227058$ & $5780112$ & $152604335$ & $4190265424$\\
\hline
$C_7$ & {} & {} & {} & $294$ & $6536$ & $254063$ & $7885801$ & $249166200$ & $7923956179$\\
\hline
$C_8$ & {} & {} & {} & $140$ & $4175$ & $206753$ & $8491679$ & $324662115$ & $12158481555$\\
\hline
$C_9$ & {} & {} & {} & {} & $2353$ & $123990$ & $7160718$ & $340745360$ & $15284159637$\\
\hline
$C_{10}$ & {} & {} & {} & {} & $722$ & $58302$ & $4634679$ & $288478215$ & $15809478819$\\
\hline
$C_{11}$ & {} & {} & {} & {} & {} & $21368$ & $2269538$ & $195270880$ & $13456339409$\\
\hline
$C_{12}$ & {} & {} & {} & {} & {} & $4439$ & $836620$ & $103755671$ & $7824793027$\\
\hline
$C_{13}$ & {} & {} & {} & {} & {} &  {} & $221987$ & $42207231$ & $5290518430$\\
\hline
$C_{14}$ & {} & {} & {} & {} & {} &  {} & $32654$ & $12701040$ & $2368530727$\\
\hline
$C_{15}$ & {} & {} & {} & {} & {} &  {} &  {} & $2624381$ & $816469677$\\
\hline
$C_{16}$ & {} & {} & {} & {} & {} &  {} &  {} & $289519$ & $206593733$\\
\hline
$C_{17}$ & {} & {} & {} & {} & {} &  {} &  {} &  {} & $34966981$\\
\hline
$C_{18}$ & {} & {} & {} & {} & {} &  {} &  {} &  {}  & $3054067$\\
\hline
$\text{total}$ & $3$ & $14$ & $110$ & $1664$ & $38610$ & $1163216$ & $42521525$ & $1821739409$ & $89423442490$\\
\hline
$\chi$ & $1$ & $2$ & $4$ & $6$ & $10$ & $16$ & $23$ & $13$ & $-96$\\
\noalign{\hrule height0.8pt}
\end{tabular}
\end{center}
\label{tab:h}
\end{table}

\begin{proof}[Proof of Theorem $\ref{thm:chi}\ \mathrm{(ii)}$]
This follows from Table \ref{tab:h}.
\end{proof}

\begin{thm}[Kontsevich {\cite{kontsevich1, kontsevich2}}]\label{thm:kontsevich-h}
For any $k \ge 1$ and $n \ge 1$, there exists an isomorphism 
\[PH_k (\mathfrak{h}^+_{\infty,1})^{\mathrm{Sp}}_{2n}
\cong 
H^{2n-k} (\mathrm{Out}\,F_{n+1};\Q).\]
\label{thm:kh}
\end{thm}

\begin{prop}
The weight generating function, denoted by $h(t)$, for the
$\mathrm{Sp}$-invariant stable homology group $H_*(\mathfrak{h}^+_{\infty,1})^{\mathrm{Sp}}$
is given by
$$
h(t)=\prod_{n=2}^\infty (1-t^{2n-2})^{- e(\mathrm{Out}\, F_n)}
$$
where $e(\mathrm{Out}\, F_n)$ denotes the {\it integral} Euler characteristic of $\mathrm{Out}\, F_n$.
\label{prop:ht}
\end{prop}

\begin{proof}
This follows from Theorem \ref{thm:kh} and Proposition \ref{prop:wp}.
\end{proof}

\begin{proof}[Proof of Theorem $\ref{thm:chip}\ \mathrm{(ii)}$]
By Theorem \ref{thm:chi} (ii), we see that the weight generating function
$$
h(t)=\sum_{w=0}^\infty \chi(H_*(\mathfrak{h}^+_{\infty,1})^{\mathrm{Sp}}_w)t^w
$$
for $H_*(\mathfrak{h}^+_{\infty,1})^{\mathrm{Sp}}$, up to weight $18$, is given by
$$
h(t)=1+t^2+2t^4+4t^6+6t^8+10t^{10}+16t^{12}+23t^{14}+13t^{16}-96t^{18}+\cdots.
$$
By applying Proposition \ref{prop:wc}, we can inductively determine the Euler characteristics of the
primitive parts, namely $e(\mathrm{Out}\, F_n)$. 
If we put
\begin{align*}
\bar{h}(t)=(1-t^2)^{-1}(1-t^4)^{-1}&(1-t^6)^{-2}(1-t^8)^{-1}\\
&(1-t^{10})^{-2}(1-t^{12})^{-1}(1-t^{14})^{-1}(1-t^{16})^{21}(1-t^{18})^{124},
\end{align*}
then we see that
$$
h(t)-\bar{h}(t)\equiv 0 
\bmod t^{19}.
$$
By Proposition \ref{prop:ht}, we can now conclude that $e(\mathrm{Out}\, F_n)=1,1,2,1,2,1,1,-21,-124$ for $n=2,3,\ldots,10$, respectively.
The result is depicted in the fourth row of Table \ref{tab:hn}.
\end{proof}

\begin{table}[h]
\caption{$\text{Numbers of new generators for $H_*(\mathfrak{h}^+_{\infty,1})^{\mathrm{Sp}}_w$}$}
\begin{center}
\begin{tabular}{|c|r|r|r|r|r|r|r|r|r|}
\noalign{\hrule height0.8pt}
\hline
 $w$ & $2$ & $4$ & $6$ & $8$ & $10$ & $12$ & $14$ & $16$ & $18$\\
\hline
$\chi$ & $1$ & $2$ & $4$ & $6$ & $10$ & $16$ & $23$ & $13$  & $-96$\\
\hline
$\text{$\chi$ of lower terms}$ & $0$ & $1$ & $2$ & $5$ & $8$ & $15$ & $22$ & $34$  & $28$\\
\hline
$\text{$\chi$ of primitive part}$ & $1$ & $1$ & $2$ & $1$ & $2$ & $1$ & $1$ & $-21$  & $-124$\\
\noalign{\hrule height0.8pt}
\end{tabular}
\end{center}
\label{tab:hn}
\end{table}

Thus we see that there are many {\it odd} dimensional non-trivial rational
cohomology classes of $\mathrm{Out}\,F_9$ as well as $\mathrm{Out}\,F_{10}$.
Before this result, very few results have been known about the rational cohomology 
group of 
$\mathrm{Out}\,F_{n}$.
As for the cases $n\leq 6$, by the works of Hatcher and Vogtmann \cite{hv} 
as well as Ohashi \cite{o}, the only non-trivial cohomology groups
 are
$H^4(\mathrm{Out}\,F_{4};\Q)\cong\Q$ and $H^8(\mathrm{Out}\,F_{6};\Q)\cong\Q$.
On the other hand, by making use of the trace maps introduced in \cite{morita93}
which give a large abelian quotient of $\mathfrak{h}^+_{g,1}$,
the first named author defined many rational homology classes
of $\mathrm{Out}\, F_n$ in \cite{morita99}\cite{morita06},
the most important classes being a series of homology classes
$$
\mu_k\in H_{4k}(\mathrm{Out}\, F_{2k+2};\Q) \quad (k=1,2,\ldots).
$$
It was conjectured in \cite{morita99} that these will be all non-trivial.
However, at present only the first three classes are known to be non-trivial,
$\mu_1$ in \cite{morita99}, $\mu_2$ by Conant and Vogtmann \cite{cov} 
and $\mu_3$ by Gray \cite{gr}. Conant and Vogtmann also gave
a geometric construction of many homology classes in the framework
of the {\it Outer Space} of Culler and Vogtmann \cite{cuv}.

As already mentioned in the introduction, recently, Conant, Kassabov and Vogtmann \cite{ckv} proved 
a remarkable result about the structure of $\mathfrak{h}_{g,1}$.
They found a deep connection with the theory of elliptic modular forms
by which they show the existence of a large new abelianization beyond the trace maps.
In particular, they defined many new cohomology classes 
in $H^2_c(\widehat{\mathfrak{h}}^+_{\infty,1})^{\mathrm{Sp}}_{2w}$ whenever
the dimension of the cusp forms of some weight $w$ is larger than $1$,
the first one being $w=24$.
These classes then produce, by Theorem \ref{thm:kh}, rational homology classes 
of $\mathrm{Out}\, F_n$ the first of which
lies in $H_{46}(\mathrm{Out}\, F_{25};\Q)$.

Now we go back to the case of $\mathrm{Out}\, F_7$ 
which is the unknown case with the smallest rank.
By our result Theorem \ref{thm:chip}\ $\mathrm{(ii)}$,
the Euler characteristic of this group is $1$ and it is an interesting problem
to determine whether the rational cohomology group of this group is trivial or not.
See Problem \ref{prob:st} for this.
Next we consider $\mathrm{Out}\, F_8$. Again by 
Theorem \ref{thm:chip}\ $\mathrm{(ii)}$, $e(\mathrm{Out}\, F_8)=1$.
On the other hand, Gray \cite{gr} proved that $\mu_3\not=0\in H_{12}(\mathrm{Out}\, F_8;\Q)$.
It follows that there exists at least one {\it odd} dimensional rational homology class.
Here we propose a candidate of such a class
in the following proposition ($\gamma_1\in H_{11}(\mathrm{Out}\, F_8;\Q)$ is our candidate).
For this, we use the summands 
$H_1(\mathfrak{h}^+_{\infty,1})_{2k+4}\supset [2k+1,1]_{\mathrm{Sp}}\ (k=1,2,\ldots)$
which are part of the new abelianizations found by 
Conant, Kassabov and Vogtmann cited above. 
By an explicit computation motivated by their result, 
we have proved that $H_1(\mathfrak{h}^+_{g,1})_{6}\cong [31]_{\mathrm{Sp}}$.

\begin{prop}
For any $k\geq 1$, we have an isomorphism
$$
\left([31]_{\mathrm{Sp}}\otimes [2k+1]_{\mathrm{Sp}}\otimes[2k+3]_{\mathrm{Sp}}\right)^{\mathrm{Sp}}
\cong\Q
$$
so that we obtain a series of (co)homology classes
$$
\gamma_k\in PH_c^{3}(\widehat{\mathfrak{h}}_{\infty,1})^{\mathrm{Sp}}_{4k+10}
\overset{\mathrm{Kontsevich}}
{\cong} H_{4k+7}(\mathrm{Out}\,F_{2k+6};\Q) \quad (k=1,2,\ldots).
$$
\end{prop}

\begin{proof}
By the Littlewood-Richardson rule, it is easy to see that
$$
[2k+1]_{\mathrm{Sp}}\otimes [2k+3]_{\mathrm{Sp}}
\cong 
[4k+4]_{\mathrm{GL}}\oplus [4k+3,1]_{\mathrm{GL}}\oplus \cdots\oplus
[2k+3,2k+1]_{\mathrm{GL}}.
$$
On the other hand, among the ${\mathrm{Sp}}$-irreducible decompositions of the
${\mathrm{GL}}$-irreducible summands on the
right hand side, only the last one $[2k+3,2k+1]_{\mathrm{GL}}$
contains $[31]_{\mathrm{Sp}}$ and the multiplicity is $1$.
The claim follows.
\end{proof}

\begin{conj}
These classes are all non-trivial. In particular, $H_3(\mathfrak{h}_{\infty,1})$ is {\it infinite}
dimensional.
\end{conj}

Also if we combine the trace components $[2k+1]_{\mathrm{Sp}}$
with the new components  $[2\ell+1,1]_{\mathrm{Sp}},...$,
we obtain huge amount of (co)homology classes
of $\mathfrak{h}_{\infty,1}$.

Next we consider the problem of comparison between the {\it rational} and
the {\it integral\/} Euler characteristics of $\mathrm{Out}\,F_{n}$.
The second row of Table \ref{tab:chie} 
is taken from Smillie and Vogtmann \cite{sv} where we write the values to the second decimal places
and the third row is our Theorem \ref{thm:chip}\ $\mathrm{(ii)}$.

\begin{table}[h]
\caption{$\text{$\chi$ versus $e$ for $\mathrm{Out}\,F_n$}$}
\begin{center}
\begin{tabular}{|c|c|c|c|c|c|c|c|c|c|}
\noalign{\hrule height0.8pt}
\hline
 $n$ & $2$ & $3$ & $4$ & $5$ & $6$ & $7$ & $8$ & $9$ & $10$\\
\hline
$\chi$ & $-0.04$ & $-0.02$ & $-0.02$ & $-0.06$ & $-0.20$ & $-0.87$ & $-4.58$ & $-28.52$  & $-205.83$\\
\hline
$e$ & $1$ & $1$ & $2$ & $1$ & $2$ & $1$ & $1$ & $-21$  & $-124$\\
\noalign{\hrule height0.8pt}
\end{tabular}
\end{center}
\label{tab:chie}
\end{table}

\begin{problem}
Study the relation between $\chi(\mathrm{Out}\,F_{n})$ and $e(\mathrm{Out}\,F_{n})$.
\end{problem}


We refer to the book \cite{f} edited by Farb,
in particular Bridson and Vogtmann \cite{bv},
as well as Farb \cite{f1}
for various problems concerning
$\mathrm{Out}\, F_n$, mapping class groups,
$\mathrm{GL}(n,\Z)$ and other related groups.

\section{The case of $\mathfrak{a}_g$ and the moduli spaces of curves}\label{sec:a}

Finally we consider the case of $\mathfrak{a}_g$.
From the point of view of computations,
this case is the most heavy one
among the three Lie algebras as can be seen 
by comparing the size of the numbers in the former
two tables Tables \ref{tab:c} and \ref{tab:h} with the present one
depicted in Table \ref{tab:a}. 

As was already mentioned in Section $1$, in \cite{mss2a} we determine the values $\chi(H_*(\mathfrak{a}^+_{\infty})^{\mathrm{Sp}}_{w})$
for all $w\leq 500$ by a completely different argument which makes use of a formula of
Gorsky \cite{go2}
for the equivariant Euler characteristics of the moduli spaces of curves $\mathbf{M}_g^m$. 
We confirm that the two values for $w\leq 16$ are the same.
We think that this coincidence serves as a strong evidence for the
accuracy of our computations in the other two cases $\mathfrak{c}_\infty, \mathfrak{h}_{\infty,1}$.
We mention that the existence of the formula of Gorsky
depends heavily on the fact that the totality of $\mathbf{M}_g^m$ for various
$g$ and $m$ makes a beautiful unified world. 
It seems unlikely that similar formulas will be
found in the other two cases, at least in a near future.


\begin{table}[h]
\caption{$\text{Case of $\mathfrak{a}_g$}$}
\begin{center}
\begin{tabular}{|c|r|r|r|r|r|r|r|r|}
\noalign{\hrule height0.8pt}
\hline
$w$ & $2$ & $4$ & $6$ & $8$ & $10$ & $12$ & $14$ & $16$\\
\hline
$C_1$ & $1$ & $2$ & $17$ & $88$ & $897$ & $9562$ & $127071$ & $1912970$\\
\hline
$C_2$ & $3$ & $8$ & $111$ & $1146$ & $14735$ & $212965$ & $3483545$ & $63522967$\\
\hline
$C_3$ & {} & $18$ & $289$ & $5561$ & $99285$ & $1918401$ & $39558275$ & $880137499$ \\
\hline
$C_4$ & {} & $17$ & $403$ & $13653$ & $366878$ & $9590016$ & $253890290$ & $6966037951$ \\
\hline
$C_5$ & {} & {} & $320$ & $19138$ & $827528$ & $30225682$ & $1047033554$ & $35904134757$ \\
\hline
$C_6$ & {} & {} & $124$ & $15860$ & $1193367$ & $63894814$ & $2967604968$ & $129283963277$ \\
\hline
$C_7$ & {} & {} & {} & $7466$ & $1111456$ & $93211250$ & $6001387476$ & $339000966002$ \\
\hline
$C_8$ & {} & {} & {} & $1618$ & $651577$ & $94398768$ & $8825700683$ & $663520078156$ \\
\hline
$C_9$ & {} & {} & {} & {} & $220905$ & $65356859$ & $9484791225$ & $982320832329$ \\
\hline
$C_{10}$ & {} & {} & {} & {} & $33564$ & $29594121$  & $7384704777$ & $1104356533575$ \\
\hline
$C_{11}$ & {} & {} & {} & {} & {} & $7925093$ & $4061192184$ & $938047301852$ \\
\hline
$C_{12}$ & {} & {} & {} & {} & {} & $956263$ & $1497800877$  & $592700462357$\\
\hline
$C_{13}$ & {} & {} & {} & {} & {} &  {} & $332831365$ & $270228006160$ \\
\hline
$C_{14}$ & {} & {} & {} & {} & {} & {} & $33736198$ & $84077896041$ \\
\hline
$C_{15}$ & {} & {} & {} & {} & {} &  {} &  {}  & $15987868100$ \\
\hline
$C_{16}$ & {} & {} & {} & {} & {} & {} &   {}  & $1402665692$ \\
\hline
$\text{total}$ & $4$ & $45$ & $1264$ & $64530$ & $4520192$ & $397293794$ & $41933842488$ & $5164742319685$\\
\hline
$\chi$ & $2$ & $5$ & $12$ & $24$ & $50$ & $100$ & $188$ & $347$ \\
\noalign{\hrule height0.8pt}
\end{tabular}
\end{center}
\label{tab:a}
\end{table}

\begin{proof}[Proof of Theorem $\ref{thm:chi}\ \mathrm{(iii)}$]
This follows from Table \ref{tab:a}.
\end{proof}

\begin{thm}[Kontsevich {\cite{kontsevich1, kontsevich2}}]\label{thm:kontsevich-a}
For any $k \ge 1$ and $n \ge 1$, there exists an isomorphism 
\[PH_k (\mathfrak{a}^+_{\infty})^{\mathrm{Sp}}_{2n}
\cong \bigoplus_{\begin{subarray}{c}
2g-2+m=n\\ m>0
\end{subarray}}
H^{2n-k} (\mathbf{M}_g^m;\Q)^{\mathfrak{S}_m}.\]
\label{thm:ka}
\end{thm}

\begin{prop}
The weight generating function, denoted by $a(t)$, for the
$\mathrm{Sp}$-invariant stable homology group $H_*(\mathfrak{a}^+_{\infty})^{\mathrm{Sp}}$
is given by
$$
a(t)=\prod_{n=1}^\infty (1-t^{2n})^{- a^{\mathrm{pr}}_{2n}}
$$
where 
$$
a^{\mathrm{pr}}_{2n}=
\sum_{\begin{subarray}{c}
2g-2+m=n\\ m>0
\end{subarray}}
e(\mathbf{M}_g^m/\mathfrak{S}_m).
$$
\label{prop:at}
\end{prop}

\begin{proof}
Note that $H^\ast (\mathbf{M}_g^m;\mathbb{Q})^{\mathfrak{S}_{m}} 
\cong H^\ast (\mathbf{M}_g^m/\mathfrak{S}_{m};\mathbb{Q})$.  Our claim 
follows from Theorem \ref{thm:ka} and Proposition \ref{prop:wp}.
\end{proof}

\begin{proof}[Proof of Theorem $\ref{thm:chip}\ \mathrm{(iii)}$]
By Theorem \ref{thm:chi} (iii), we see that the weight generating function
$$
a(t)=\sum_{w=0}^\infty \chi(H_*(\mathfrak{a}^+_{\infty})^{\mathrm{Sp}}_w)t^w
$$
for $H_*(\mathfrak{a}^+_{\infty})^{\mathrm{Sp}}$, up to weight $16$, is given by
$$
a(t)=1+2t^2+5t^4+12t^6+24t^8+50t^{10}+100t^{12}+188t^{14}+347t^{16}+\cdots.
$$
By applying Proposition \ref{prop:wc}, we can inductively determine the Euler characteristics of the
primitive parts, namely $e(\mathbf{M}_g^m/\mathfrak{S}_m)$. 
If we put
\begin{align*}
\bar{a}(t)=(1-t^2)^{-2}(1-t^4)^{-2}&(1-t^6)^{-4}(1-t^8)^{-2}\\
&(1-t^{10})^{-6}(1-t^{12})^{-6}(1-t^{14})^{-6}(1-t^{16})^{-1},
\end{align*}
then we see that
$$
a(t)-\bar{a}(t)\equiv 0 
\bmod t^{17}.
$$
By Proposition \ref{prop:at}, we can now conclude that 
$a^{\mathrm{pr}}_{2n}=2,2,4,2,6,6,6,1$ for $n=1,2,\ldots,8$, respectively.
The result is depicted in the fourth row of Table \ref{tab:an} as well as
the following proposition \ref{prop:ass}.
\end{proof}

\begin{table}[h]
\caption{$\text{Numbers of new generators for $H_*(\mathfrak{a}^+_{\infty})^{\mathrm{Sp}}$}$}
\label{tab:aan}
\begin{center}
\begin{tabular}{|c|r|r|r|r|r|r|r|r|}
\noalign{\hrule height0.8pt}
\hline
$w$ & $2$ & $4$ & $6$ & $8$ & $10$ & $12$ & $14$ & $16$\\
\hline
$\chi$ & $2$ & $5$ & $12$ & $24$ & $50$ & $100$ & $188$ & $347$ \\
\hline
$\text{$\chi$ of lower terms}$ & $0$ & $3$ & $8$ & $22$ & $44$ & $94$ & $182$ & $346$ \\
\hline
$\text{$\chi$ of primitive part}$ & $2$ & $2$ & $4$ & $2$ & $6$ & $6$ & $6$ & $1$ \\
\noalign{\hrule height0.8pt}
\end{tabular}
\end{center}
\label{tab:an}
\end{table}

\begin{prop}
We have the following equalities.
\begin{align*}
\mathrm{(1)}\quad &e(\mathbf{M}_0^3/\mathfrak{S}_3)+e(\mathbf{M}_1^1) =2\\
\mathrm{(2)}\quad &e(\mathbf{M}_0^4/\mathfrak{S}_4)+e(\mathbf{M}_1^2/\mathfrak{S}_2) =2\\
\mathrm{(3)}\quad &e(\mathbf{M}_0^5/\mathfrak{S}_5)+e(\mathbf{M}_1^3/\mathfrak{S}_3)+e(\mathbf{M}_2^1) =4\\
\mathrm{(4)}\quad &e(\mathbf{M}_0^6/\mathfrak{S}_6)+e(\mathbf{M}_1^4/\mathfrak{S}_4)+e(\mathbf{M}_2^2/\mathfrak{S}_2) =2\\
\mathrm{(5)}\quad &e(\mathbf{M}_0^7/\mathfrak{S}_7)+e(\mathbf{M}_1^5/\mathfrak{S}_5)+e(\mathbf{M}_2^3/\mathfrak{S}_3)+e(\mathbf{M}_3^1) =6\\
 \mathrm{(6)}\quad &e(\mathbf{M}_0^8/\mathfrak{S}_8)+e(\mathbf{M}_1^6/\mathfrak{S}_6)+e(\mathbf{M}_2^4/\mathfrak{S}_4)+e(\mathbf{M}_3^2/\mathfrak{S}_2) =6\\
\mathrm{(7)}\quad &e(\mathbf{M}_0^9/\mathfrak{S}_9)+e(\mathbf{M}_1^7/\mathfrak{S}_7)+e(\mathbf{M}_2^5/\mathfrak{S}_5)+e(\mathbf{M}_3^3/\mathfrak{S}_3)+e(\mathbf{M}_4^1) =6\\
 \mathrm{(8)}\quad &e(\mathbf{M}_0^{10}/\mathfrak{S}_{10})+e(\mathbf{M}_1^8/\mathfrak{S}_8)+e(\mathbf{M}_2^6/\mathfrak{S}_6)+e(\mathbf{M}_3^4/\mathfrak{S}_4)+e(\mathbf{M}_4^2/\mathfrak{S}_2) =1.
\end{align*}
\label{prop:ass}
\end{prop}

Now we check that our result above is consistent with the known results.
By Getzler \cite{getz99}, $e(\mathbf{M}_0^m/\mathfrak{S}_m)=1$ for all $m\geq 3$.
Also he determined the $\mathfrak{S}_m$-equivariant Serre polynomial of $\mathbf{M}_1^m$. 
In particular, he obtained the following formula
\begin{align*}
\sum_{m=1}^\infty e(\mathbf{M}_1^m/\mathfrak{S}_m) x^m&=
(x+x^2+x^3)\frac{(1-x^4-2x^8-x^{12}+x^{16})}{(1-x^8)(1-x^{12})}\\
&=x+x^2+x^3-x^5-x^6-x^7-x^9-x^{10}-x^{11}-x^{13}-x^{14}-\cdots
\end{align*}
so that $e(\mathbf{M}_1^m/\mathfrak{S}_m)=1,1,1,0,-1,-1,-1,0,-1$ for 
$m=1,\ldots,9$.
It is well known that $e(\mathbf{M}_2^1)=2$ and Getzler proved $e(\mathbf{M}_2^2)=1$.
The fifth equality $\mathrm{(5)}$ (case of $w=10$) was first proved by Getzler and Kapranov \cite{gk},
and then in \cite{getz98} it was shown that $e(\mathbf{M}_2^3/\mathfrak{S}_3)=0$
by using the result of
Looijenga \cite{loo} determining $H^*(\mathbf{M}_3;\Q), H^*(\mathbf{M}^1_3;\Q)$,
especially $e(\mathbf{M}^1_3)=6$.
Tommasi \cite{tom05th} (see also \cite{tom05}\cite{tom06}) determined 
$H^*(\mathbf{M}_4,\Q)$ as well as the
equivariant Hodge Euler characteristics of $\mathbf{M}_2^4$ and $\mathbf{M}_3^2$
and in particular $e(\mathbf{M}_2^4/\mathfrak{S}_4)=1$ and $e(\mathbf{M}_3^2/\mathfrak{S}_2)=5$.
The sixth equality $\mathrm{(6)}$ is consistent with these results.
Next we consider the seventh equality. By Harer and Zagier \cite{hz}, $e(\mathbf{M}_4^1)=2$.
Gorsky \cite{go1} (Theorem $2$) extended the work of Getzler and obtained a formula for
the $\mathfrak{S}_m$-equivariant Euler characteristic for $\mathbf{M}_2^m$.
More precisely he obtained a formula for the generating function
$\sum_{m=0}^\infty e^{\mathfrak{S}_m}(\mathbf{M}_2^m) t^m$ in terms of
Newton's power sum polynomials.
He then made a computer computation and determined $e^{\mathfrak{S}_m}(\mathbf{M}_2^m)$
explicitly for all $m\leq 4$ which coincide with the former results of Getzler and Tommasi cited above.
By making use of our Method (I) and Method (V) described in Section $4$, we 
extended Gorsky's computation to obtain closed formulas for the cases $5\leq m\leq 35$.
Here we describe the results for $m=5,6,7,8$.
\begin{align*}
e^{\mathfrak{S}_5}(\mathbf{M}_2^5)&=2[5]-2[32]+2[41]\\
e^{\mathfrak{S}_6}(\mathbf{M}_2^6)&=-2[3^2]+2[51]-3[2^3]-2[321]+2[41^2]+[2^21^2]-[21^4]-[1^6]\\
e^{\mathfrak{S}_7}(\mathbf{M}_2^7)&=-2[7]-2[43]+2[52]-2[61]-4[32^2]-2[3^21]+4[421]+2[51^2]-\\
&2[2^31]+4[321^2]+2[41^3]-2[21^5]-2[1^7]\\
e^{\mathfrak{S}_8}(\mathbf{M}_2^8)&=-[8]-3[4^2]-4[53]+2[62]-3[71]-7[3^22]+4[42^2]-7[431]+2[521]-\\
&7[61^2]-4[2^4]-3[32^21]-2[3^21^2]-[421^2]-3[51^3]-7[2^31^2]+2[41^4]-\\&3[2^21^4]-
5[21^6]-3[1^8]
\end{align*}
It is amusing to calculate the dimensions of the above expression which give
the integral Euler characteristics of $\mathbf{M}_2^m$. The results are
$0,-24$ for $m=5, 6$
and $(-1)^{m+1} (m+1)!/240$ for $m\geq 7$ which coincide with the known values obtained by
Harer and Zagier \cite{hz}.
On the other hand, as for the coefficients of the trivial representation $[m]$, it can be shown that
Gorsky's formula implies
\begin{align*}
\sum_{m=0}^\infty &e(\mathbf{M}_2^m/\mathfrak{S}_m) x^m=
-\frac{1}{240}(1+x)^{-2}-\frac{1}{240}(1+x)^{6}(1+x^2)^{-4}+\frac{1}{12}(1+x)^{2}(1+x^2)^{-2}\\
&-\frac{1}{12}(1+x)^{4}(1+x^3)^{-2}-\frac{1}{8}(1+x)^{2}(1+x^2)^{2}(1+x^4)^{-2}+\frac{2}{5}(1+x)^{3}(1+x^5)^{-1}\\
&-\frac{1}{12}(1+x^2)^{2}(1+x^3)^{2}(1+x^6)^{-2}
+\frac{1}{6}(1+x)^{2}(1+x^2)(1+x^6)^{-1}\\
&+\frac{1}{4}(1+x)^{2}(1+x^4)(1+x^8)^{-1}
+\frac{2}{5}(1+x)(1+x^2)(1+x^5)(1+x^{10})^{-1}\\
&=1+2x+x^2+x^4+2x^5-2x^7-x^8-x^{10}+2x^{12}+2x^{13}-3x^{14}+\cdots
\end{align*}
so that we find $e(\mathbf{M}_2^m/\mathfrak{S}_m)=2,0,-2,-1,0,-1,0,2,2,-3,\ldots$ for 
$m=5,6,\ldots,$ $14,\ldots$.
By substituting these known values in the seventh equality $\mathrm{(7)}$, we obtain
$e(\mathbf{M}_3^3/\mathfrak{S}_3)=2$. This should be consistent with 
the work of Bergstr\"om \cite{b} determining $H^*(\overline{\mathbf{M}}_3^m;\Q)$
as well as $H^*(\mathbf{M}_3^m;\Q)$
as an $\mathfrak{S}_m$-module for all $m\leq 5$, although the latter is not described explicitly.
From the eighth equality $\mathrm{(8)}$ together with the above results, we conclude
$
e(\mathbf{M}_3^4/\mathfrak{S}_4)+e(\mathbf{M}_4^2/\mathfrak{S}_2)=0.
$
Now Gorsky \cite{go2} (Theorem $3$) extended his own result cited above to obtain
a formula for the equivariant Euler characteristics of {\it all} the moduli spaces $\mathbf{M}_g^m$
again in terms of Newton power sum polynomials. It is an extensive generalization
of the results of Getzler cited above
as well as the formula of Harer and Zagier \cite{hz} for the {\it integral} Euler characteristics of
the moduli space of curves.
By using Gorsky's formula, we obtain explicit closed formulas for the equivariant
Euler characteristic $e^{\mathfrak{S}_m} (\mathbf{M}_g^m)$ as well as
the generating functions
$$
\sum_{m=0}^\infty e(\mathbf{M}_g^m/\mathfrak{S}_m) x^m
$$
for $g\leq 125$. Here we describe the values for $g=3,4; m\leq 10$.
\begin{align*}
&e(\mathbf{M}_3^m/\mathfrak{S}_m)=3,6,5,2,0,0,1,0,-1,2,2 \quad (m=0,1,\ldots,10)\\
&e(\mathbf{M}_4^m/\mathfrak{S}_m)=2,2,0,2,0,2,10,6,-19,-12,34  \quad (m=0,1,\ldots,10).
\end{align*}
In particular $e(\mathbf{M}_3^4/\mathfrak{S}_4)=e(\mathbf{M}_4^2/\mathfrak{S}_2)=0$.
Thus we find that our results are completely consistent with known results in algebraic geometry,
or can be deduced from them by explicit computations. In our forthcoming paper
\cite{mss2a}, we will further discuss these formulas.

We mention that the integral Euler characteristic of the moduli space $\mathbf{M}_g^m$,
rather than its quotient $\mathbf{M}_g^m/\mathfrak{S}_m$, is known by Harer and Zagier \cite{hz}
as well as Bini and Harer \cite{bh} up to certain values of $g,m$.
In particular, $e(\mathbf{M}_3^4)=4, e(\mathbf{M}_4^2)=-2$.

In our paper \cite{mss1}, we proved that the stable abelianization of $\mathfrak{a}_g$
is trivial, namely $H_1(\mathfrak{a}^+_\infty)^{\mathrm{Sp}}=0$ and
deduced from it the vanishing result, $H^{4g-5}(\Mg;\Q)=0$,
of the top rational cohomology group of the mapping class group $\Mg$
of a closed oriented surface of genus $g$ for all $g\geq 2$.
See also Church, Farb and Putman \cite{cfp1}.

\begin{problem}
For each $k\geq 2$, determine whether the $\mathrm{Sp}$-invariant stable homology group
$H_k(\mathfrak{a}^+_{\infty})^{\mathrm{Sp}}$
is {\it finite} dimensional or not, especially for the low
values $k=2,3,\ldots$.
\end{problem}

\begin{remark}
Kontsevich made a conjecture in \cite{kontsevich1} that
the stable homology of each of the infinite dimensional Lie
algebras he considered is {\it finite} dimensional in each degree.
We would like to mention the following implication of this conjecture
in the case of $\mathfrak{a}_g$. Namely,
if $H_k(\mathfrak{a}^+_{\infty})^{\mathrm{Sp}}$ is finite dimensional,
then
$$
 H_{4g-4-k}(\mathbf{M}_{g};\Q) =0\;
\text{for all $g$ but finitely many possible exceptions}.
$$
This is because, if the assumption is valid, then $PH_c^k$ is also finite dimensional
so that $H_{4g-4+2m-k}(\mathbf{M}^m_{g};\Q)^{\mathfrak{S}_m}=0$ for all
$g$ and $m>0$ except for finitely many values. If we put $m=1$ here, we see that 
$H_{4g-2-k}(\mathbf{M}^1_{g};\Q)=0$ for all $g$ but finitely many exceptions.
Now if $H_{4g-4-k}(\mathbf{M}_g;\Q)\not=0$ for some $g\geq 2$,
then as was proved in \cite{morita87} that
the homomorphism 
$$
H^{*}(\mathbf{M}_g;\Q) \xrightarrow{p^*}
H^{*}(\mathbf{M}^1_{g};\Q)  \xrightarrow{\cup\ e}
H^{*+2}(\mathbf{M}^1_{g};\Q)
$$
is injective for any $g\geq 2$ where $e\in H^2(\mathbf{M}^1_{g};\Q)$ denotes
the Euler class. It follows that $H_{4g-2-k}(\mathbf{M}^1_{g};\Q)\not=0$
for such $g$. The claim follows. 

Looijenga \cite{loo}
determined the rational homology groups of both of $\mathbf{M}_3, \mathbf{M}_3^1$ 
and in particular
he found an unstable cohomology class in $H^6(\mathbf{M}_3;\Q)$
which is the first known unstable class of the moduli spaces
{\it without} punctures. He also showed that
$H^8(\mathbf{M}_3^1;\Q)\cong\Q, H^6(\mathbf{M}_3^1;\Q)\cong\Q$.
It follows that $PH_2(\mathfrak{a}^+_\infty)^{\mathrm{Sp}}_{10}\cong\Q$
and 
$PH_4(\mathfrak{a}^+_\infty)^{\mathrm{Sp}}_{10}\supset \Q$.
It is a very important problem to determine whether $H_2(\mathfrak{a}^+_{\infty})^{\mathrm{Sp}}$
is finite dimensional or not.

Recently Church, Farb and Putman \cite{cfp2} proposed a new stability conjecture about the unstable
cohomology of $\mathrm{SL}(n,\Z), \mathrm{Aut}\, F_n$
and the mapping class groups. In the case of the mapping class groups, our argument above
implies the following. Namely, if the above conjecture
of Kontsevich is true, then their conjecture also
holds in the form that all the groups are trivial.
\end{remark}

\section{Dimensions of the $\mathrm{Sp}$-invariant subspaces
$\mathfrak{h}_{g,1}(2k)^{\mathrm{Sp}}$}\label{sec:i}

As an application of our consideration, we obtain a complete description of 
how the $\mathrm{Sp}$-invariant subspaces $\mathfrak{h}_{g,1}(k)^{\mathrm{Sp}}$
degenerate with respect to $g$. It turns out that this degeneration is perfectly
compatible with the orthogonal direct sum decomposition of $\mathfrak{h}_{g,1}(k)^{\mathrm{Sp}}$
with respect to the canonical metric on it introduced in \cite{morita11}.
The results for all $k\leq 20$ are depicted in
Tables \ref{tab:i1} and \ref{tab:i2}. In the tables, the symbol $*$ denotes that
the dimension of the $\mathrm{Sp}$-invariant subspaces {\it stabilizes} there
with respect to the genus $g$. In general, we can show that the stable range
is given by
\begin{align*}
&\dim\, \mathfrak{h}_{k,1}(2k)^{\mathrm{Sp}}=\dim\, \mathfrak{h}_{k+1,1}(2k)^{\mathrm{Sp}}=\cdots\quad (\text{$k$ odd})\\
&\dim\, \mathfrak{h}_{k-1,1}(2k)^{\mathrm{Sp}}=\dim\, \mathfrak{h}_{k,1}(2k)^{\mathrm{Sp}}=\cdots\quad (\text{$k$ even}).
\end{align*}

\begin{table}[h]
\caption{$\text{Dimensions of $\mathfrak{h}_{g,1}(2k)^{\mathrm{Sp}}$\  $\mathrm{(i)}$}$}
\begin{center}
\begin{tabular}{|c|r|r|r|r|r|r|r|r|r|}
\noalign{\hrule height0.8pt}
\hfil {} & $g=1$ & $g=2$ & $g=3$ & $g=4$ & $g=5$  \\
\hline
$\mathfrak{h}_{g,1}(2)^{\mathrm{Sp}}$ & $1*$ & $1$ & $1$ & $1$ & $1$ \\
\hline
$\mathfrak{h}_{g,1}(4)^{\mathrm{Sp}}$ & $0*$ & $0$ & $0$ & $0$ & $0$ \\
\hline
$\mathfrak{h}_{g,1}(6)^{\mathrm{Sp}}$ & $1$ & $4$ & $5*$ & $5$ & $5$ \\
\hline
$\mathfrak{h}_{g,1}(8)^{\mathrm{Sp}}$ & $0$ & $2$ & $3*$ & $3$ & $3$ \\
\hline
$\mathfrak{h}_{g,1}(10)^{\mathrm{Sp}}$ & $3$ & $51$ & $97$ & $107$ & $108*$ \\                                                               
\hline
$\mathfrak{h}_{g,1}(12)^{\mathrm{Sp}}$ & $0$ & $190$ & $544$ & $643$ & $650*$ \\
\hline
$\mathfrak{h}_{g,1}(14)^{\mathrm{Sp}}$ & $11$ & $1691$ & $6471$ & $8505$ & $8795$ \\
\hline
$\mathfrak{h}_{g,1}(16)^{\mathrm{Sp}}$ & $10$ & $11842$ & $69544$ & $104190$ & $110610$ \\
\hline
$\mathfrak{h}_{g,1}(18)^{\mathrm{Sp}}$ & $57$ & $100908$ & $888099$ & $1548984$ & 
$1710798$ \\
\hline
$\mathfrak{h}_{g,1}(20)^{\mathrm{Sp}}$ & $108$ & $869798$ & $12057806$ & $25062360$ & $29129790$ \\
\noalign{\hrule height0.8pt}
\end{tabular}
\end{center}
\label{tab:i1}
\end{table}

\begin{table}[h]
\caption{$\text{Dimensions of $\mathfrak{h}_{g,1}(2k)^{\mathrm{Sp}}$\  $\mathrm{(ii)}$}$}
\begin{center}
\begin{tabular}{|c|r|r|r|r|r|r|r|r|r|}
\noalign{\hrule height0.8pt}
\hfil {} & $g=6$ & $g=7$ & $g=8$ & $g=9$   \\
\hline
$\mathfrak{h}_{g,1}(14)^{\mathrm{Sp}}$ & $8816$ & $ 8817*$ & $8817$ & $8817$  \\
\hline
$\mathfrak{h}_{g,1}(16)^{\mathrm{Sp}}$ & $111131$ & $111148*$ & $111148$ & $111148$ \\
\hline
$\mathfrak{h}_{g,1}(18)^{\mathrm{Sp}}$ & $1728591$ & $1729620$ & $1729656$ & $1729657*$ \\
\hline
$\mathfrak{h}_{g,1}(20)^{\mathrm{Sp}}$ & $29688027$ & $29728348$ & $29729957$ & $29729988*$ \\
\noalign{\hrule height0.8pt}
\end{tabular}
\end{center}
\label{tab:i2}
\end{table}

Details will be given in our forthcoming paper \cite{mss3} where we will discuss how the Lie bracket operation
on $\mathfrak{h}_{g,1}^{\mathrm{Sp}}$ is related to the above description of the orthogonal direct sum decomposition
as well as
the degeneration. This should be important in the investigation of the {\it arithmetic} mapping class group.

\section{Concluding remarks and problems}\label{sec:r}

In this section, we discuss differences between the three cases 
$\mathfrak{g}_\infty=\mathfrak{c}_\infty, \mathfrak{h}_{\infty,1}, \mathfrak{a}_\infty$.

\begin{remark}
First we consider the {\it lowest weight} part of the continuous
cohomology by which we mean the image of the homomorphism
$$
H_c^*(\widehat{\mathfrak{g}}_\infty(1))^{\mathrm{Sp}}\rightarrow H^*_c(\mathfrak{g}^+_\infty)^{\mathrm{Sp}}
$$
which is induced by the Lie algebra homomorphism $\mathfrak{g}_\infty \rightarrow \mathfrak{g}_\infty(1)$.
In the commutative case, this {\it lowest weight} cohomology is precisely the dual of 
$\mathcal{A}(\emptyset)\subset H_*(\mathfrak{c}^+_\infty)^{\mathrm{Sp}}$ so that it is still mysterious. 
On the other hand, in the other two cases, it is 
completely understood because
the {\it lowest weight} cohomology is precisely the totality of 
$H_0(\mathrm{Out}\,F_n;\Q)\ (n\geq 2)$ for the Lie case and 
the totality of 
$H_0(\mathbf{M}_g^m;\Q)\ (2g-2+m> 0, m\geq 1)$ for the associative case.
We mention here that the {\it lowest weight} cohomology of the Lie algebras
$\mathfrak{h}^+_{g,*}$ and $\mathfrak{h}^+_{g}$ surject onto 
the tautological algebras $\mathcal{R}^*(\mathbf{M}_{g}^1)$ and
$\mathcal{R}^*(\mathbf{M}_{g})$, respectively,  by the results of \cite{morita96},\cite{kam}.
Further, it was conjectured in \cite{morita06} that these homomorphisms are isomorphisms.
\label{rem:t}
\end{remark}

\begin{remark}
In the commutative case, it is easy to see that the homomorphism
$\mathfrak{c}^+_g\rightarrow \mathfrak{c}_g(1)=S^3H_\Q$ is nothing other than
the {\it abelianization} of the Lie algebra $ \mathfrak{c}^+_g$ because it is fairy easy to
see that the Poisson bracket $\mathfrak{c}_g(k)\otimes \mathfrak{c}_g(1)\rightarrow\mathfrak{c}_g(k+1)$
is surjective for any $k\geq 1$. Hence the lowest weight cohomology in this case is the same as
those classes which are induced from the abelianization, namely the image
of the homomorphism
$$
H^*_c\left(\widehat{H}_1(\mathfrak{c}^+_{\infty})\right)^{\mathrm{Sp}}\rightarrow 
H^*_c(\widehat{\mathfrak{c}}^+_{\infty})^{\mathrm{Sp}}.
$$
Theorem \ref{thm:ce} shows that this homomorphism is {\it not} surjective and furthermore
we conjecture that the cokernel is {\it infinitely} generated (see Conjecture \ref{conj:e}).
In the associative case, in our former paper \cite{mss1} we have determined the {\it stable} abelianization of 
$\mathfrak{a}_\infty$ which turned out to be very {\it small}. Also the associative version of the
above homomorphism is far from being surjective. In the Lie case, the known abelianization of $\mathfrak{h}^+_{\infty,1}$
turns out to be already very large by \cite{morita93}\cite{ckv}
(although the final answer is not yet known) and
many cohomology classes have been defined by making use of it.
On the other hand, we proved that $e(\mathrm{Out}\, F_{10})=-124$ (Theorem \ref{thm:chip} $\mathrm{(ii)}$)
and it seems that this number is too large to be covered by the above construction.
Because of this, we propose the following problem.
\end{remark}

\begin{problem}
Prove that the natural homomorphism
$$
H^*_c\left(\widehat{H}_1(\mathfrak{h}^+_{\infty,1})\right)^{\mathrm{Sp}}\rightarrow 
H^*_c(\widehat{\mathfrak{h}}^+_{\infty,1})^{\mathrm{Sp}}
$$
is {\it not} surjective.
\end{problem}

In view of the fact that a single irreducible piece $\mathfrak{h}_g(1)=[1^3]_{\mathrm{Sp}}$ gives
rise to the whole tautological algebra $\mathcal{R}^*(\mathbf{M}_{g})$ as already
mentioned in the preceding remark, it seems worthwhile to consider the following problem.

\begin{problem}
Let $\mathrm{Tr}(2k+1): \mathfrak{h}^+_{g,1}\rightarrow S^{2k+1} H_\Q$ be the
$(2k+1)$-st trace map defined in \cite{morita93} and let
$$
\varprojlim_{g\to\infty}\ PH^{2n}(S^{2k+1} H_\Q)^{\mathrm{Sp}}_{2n(2k+1)}\rightarrow 
PH^{2n}_c(\widehat{\mathfrak{h}}^+_{\infty,1})^{\mathrm{Sp}}_{2n(2k+1)}\cong H_{4nk}(\mathrm{Out}\,F_{2nk+n+1};\Q)
$$
be the homomorphism induced from the above {\it single} trace homomorphism. 
It is easy to see that the left hand side is non-trivial for any $n\geq 1$ and $k\geq 1$.
Determine whether the
classes in the image of this homomorphism are non-trivial or not.
\label{prob:st}
\end{problem}

The case $n=1$ corresponds to the original conjecture proposed in \cite{morita99} expecting the 
non-triviality of the classes $\mu_k$. In view of our result that $e(\mathrm{Out}\, F_7)=1$,
the case $n=2, k=1$ which asks whether the homomorphism
$
\varprojlim_{g\to\infty}\ PH^{4}(S^{3} H_\Q)^{\mathrm{Sp}}_{12}\cong\Q^2\rightarrow 
PH^{4}_c(\widehat{\mathfrak{h}}^+_{\infty,1})^{\mathrm{Sp}}_{12}\cong H_{8}(\mathrm{Out}\,F_{7};\Q)
$
is non-trivial or not, should be an important test case.

\begin{remark}
In this paper, we consider only the Euler characteristics of various chain complexes.
However we are planning to study the boundary operators as well.
In fact, we already have a proof of the non-triviality $\mu_2\not= 0$,
which was first proved by Conant and Vogtmann \cite{cov}, in our context.
Also it will be nice if one could construct cycles corresponding to the unstable
cohomology classes found by Looijenga and/or
Tommasi.
\end{remark}

\begin{remark}
We expect that there should be a close relation between the cohomology of 
$\mathrm{Out}\,F_n$ and that of $\mathrm{GL}(n,\Z)$.
For example, Elbaz-Vincent, Gangl and Soul\'e \cite{egs} recently
calculated the rational cohomology of $\mathrm{GL}(n,\Z)$ for 
$n=5,6,7$ and it will be a very interesting problem to compare 
these results with the known results about $H^*(\mathrm{Out}\, F_6;\Q)$.
Also we have a conjectural {\it geometric meaning} of the 
classes $\mu_k\in H_{4k}(\mathrm{Out}\,F_{2k+2};\Q)$.
More precisely, we expect that these classes 
can be interpreted as certain {\it secondary} classes associated with
the difference between two reasons for the Borel regulator classes $\beta_k\in H^{4k+1}(\mathrm{GL}(N,\Z);\R)$
(see Borel \cite{borel1})
to vanish first in $H^{4k+1}(\mathrm{Out}\, F_N;\R)$ by the vanishing theorem of 
Igusa \cite{i} (see also Galatius \cite{ga}) and secondly 
vanish in $H^{4k+1}(\mathrm{GL}(N^*_k,\Z);\R)$ for certain
{\it unknown} critical rank $N^*_k$ depending on $k$
(we conjecture that $N^*_k=2k+2$).
We mention that Bismut and Lott \cite{bl} proved that $\beta_k$ vanishes in
$H^{4k+1}(\mathrm{GL}(2k+1,\Z);\R)$
(we thank Christophe Soul\'e for this information) 
and the above value for $N_k^*$ is the next one after $2k+1$.
This expectation is consistent with the result of Conant and Vogtmann \cite{covs}
where they proved that $\mu_k$ vanishes after one stabilization.
\end{remark}

\begin{remark}
Our {\it ultimate} goal is to enhance our study of the cohomology of the Lie algebra $\mathfrak{h}_{g,1}$
to those of the groups $\mathcal{H}_{g,1}^{\mathrm{smooth}}, \mathcal{H}_{g,1}^{\mathrm{top}}$
of homology cobordism classes of homology cylinders over the surfaces, both in the 
 {\it smooth} as well as the {\it topological} categories and also
 the group $\mathrm{Aut}_0 F^{\mathrm{acy}}_{2g}$ which was introduced by the second named author in \cite{s}.
The group $\mathcal{H}_{g,1}^{\mathrm{smooth}}$ was introduced by
Garoufalidis and Levine \cite{gl}  as an {\it enlargement} of the mapping class group
$\mathcal{M}_{g,1}$.
 For basic facts as well as various problems about this group, we refer
 to their paper cited above.
 We would like to clarify the difference between the above three groups by 
 investigating suitable characteristic classes.
 \end{remark}

\bibliographystyle{amsplain}

\end{document}